\documentclass[12pt,a4paper]{amsart}
\usepackage[margin=2.5cm]{geometry}
\usepackage{tikz,tkz-graph}
\usepackage{float}
\usetikzlibrary{arrows.meta}
\usetikzlibrary{calc}
\usepackage{pgfplots}
\pgfplotsset{every axis/.append style={
 axis x line=middle, % put the x axis in the middle
 axis y line=middle, % put the y axis in the middle
 axis line style={-,color=blue}, % arrows on the axis
 xlabel={$x$}, % default put x on x-axis
 ylabel={$y$}, % default put y on y-axis
 }}
\pgfplotsset{compat=1.13}

\usepackage[latin1]{inputenc}
\usepackage{amsmath}
\usepackage{amsthm}
\usepackage{amsfonts}
\usepackage{amssymb}
\usepackage{mathtools}
\usepackage{enumerate}
\usepackage{graphicx}
\usepackage{hyperref}
\usepackage{nicefrac}
\usepackage{enumitem}
\usepackage{cancel}
\usepackage[all,cmtip]{xy}

\DeclareMathOperator{\Hom}{Hom}

\DeclareMathOperator{\GL}{GL}

\DeclareMathOperator{\Supp}{Supp}

\DeclareMathOperator{\QHS}{{\mathbb Q}HS^3}

\def\cS{{\mathcal S}}
\def\V{{\mathcal V}}
\def\I{{\mathcal I}}

\def\V{{V}}
\def\I{{I}}

\def\ZZ{\mathbb{Z}}

\def\CC{\mathbb{C}}

\def\QQ{\mathbb{Q}}

\def\l{\ell}

\def\cO{\mathcal{O}}

\def\div{\textrm{div}}
\def\pc{\textrm{pc}}

\def\tt{\mathbf{t}}
\def\ZK{Z_K}
\def\Zmin{Z_{\rm{min}}}
\def\val{\mathrm{val}}

\newtheorem{thm}{Theorem}[section] %CHOOSE [chapter] or [section]
\newtheorem{prop}{Proposition}[section]

\newtheorem{lemma}{Lemma}[section]
\newtheorem{def-lemma}{Definition-Lemma}[section]

\theoremstyle{remark}
\newtheorem{rem}{Remark}[section]

\theoremstyle{definition}
\newtheorem{dfn}{Definition}[section]
\newtheorem{exam}{Example}[section]

\makeatletter
\let\c@lemma\c@thm
\let\c@prop\c@thm
\let\c@propdef\c@thm
\let\c@proper\c@thm
\let\c@problem\c@thm
\let\c@conj\c@thm
\let\c@cor\c@thm
\let\c@rem\c@thm
\let\c@dfn\c@thm
\let\c@notation\c@thm
\let\c@exam\c@thm
\makeatother

\newcommand{\tX}{{\widetilde {X}}}
\newcommand{\FN}{\mathfrak{N}}
\newcommand{\FR}{\mathfrak{R}}
\newcommand{\FP}{\mathfrak{P}}

\title
{Local invariants of minimal generic curves on rational surfaces}
\date{\today}

\author[J.I. Cogolludo]{Jos{\'e} Ignacio Cogolludo-Agust{\'i}n}
\address{Departamento de Matem\'aticas, IUMA\\
Universidad de Zaragoza\\
C.~Pedro Cerbuna 12\\
50009 Zaragoza, Spain}
\email{jicogo@unizar.es}

\author[T. L\'aszl\'o]{Tam\'as L\'aszl\'o}
\address{Babe\c{s}-Bolyai University, Str. Mihail Kog\u{a}lniceanu nr. 1, 400084 Cluj-Napoca, Romania}
\email{laszlo.tamas@math.ubbcluj.ro}

\author[J. Mart\'in]{Jorge Mart\'in-Morales}
\address{Centro Universitario de la Defensa-IUMA\\
Academia General Militar\\
Ctra.~de Huesca s/n.\\
50090, Zaragoza, Spain}
\email{jorge@unizar.es}

\author[A. N\'emethi]{Andr\'as N\'emethi}
\address{Alfr\'ed R\'enyi Institute of Mathematics,
Hungarian Academy of Sciences,
Re\'altanoda utca 13-15, H-1053, Budapest, Hungary \newline
 \hspace*{4mm} ELTE - University of Budapest, Dept. of Geometry, Budapest, Hungary \newline \hspace*{4mm}
BCAM - Basque Center for Applied Math.,
Mazarredo, 14 E48009 Bilbao, Basque Country, Spain}
\email{nemethi.andras@renyi.mta.hu }

\thanks{The first and third authors are partially supported by MTM2016-76868-C2-2-P and
Gobierno de Arag{\'o}n (Grupo de referencia ``{\'A}lgebra y Geometr{\'i}a'')
cofunded by Feder 2014-2020 ``Construyendo Europa desde Arag\'on''.
The third author is also partially supported by FQM-333 ``Junta de Andaluc{\'\i}a''. \\
The second and fourth authors are supported by NKFIH Grant ``\'Elvonal (Frontier)'' KKP 126683. 
The second author was also supported by ERCEA Consolidator Grant 615655 - NMST, 
and partially by the Basque Government through the BERC 2018-2021 program and Gobierno Vasco Grant IT1094-16, 
by the Spanish Ministry of Science, Innovation and Universities: BCAM Severo Ochoa accreditation SEV-2017-0718.}

\begin{document}

\begin{abstract}
Let $(C,0)$ be a reduced curve germ in a normal surface singularity $(X,0)$.
The main goal is to recover the delta invariant $\delta(C)$
of the abstract curve $(C,0)$ from the topology of the embedding $(C,0)\subset (X,0)$.
We give explicit formulae whenever $(C,0)$ is \emph{minimal generic} and $(X,0)$ is rational
(as a continuation of~\cite{kappa1,kappa2}).

Additionally we prove that if $(X,0)$ is a quotient singularity, then $\delta(C)$ only admits
the values $r-1$ or $r$, where $r$ is the number or irreducible components of $(C,0)$.
($\delta (C)=r-1$ realizes the extremal lower bound, valid only for `ordinary $r$--tuples'.)
\end{abstract}

\maketitle

\section{Introduction}

The present note is a natural continuation of the manuscripts \cite{kappa1,kappa2}. The main setup is the following.
We fix a complex normal surface singularity $(X,0)$,
and an embedded reduced curve germ $(C,0)$ on it. We wish to connect the analytic invariants of the abstract curve $(C,0)$,
eg. its delta invariant $\delta(C)$, with the embedded geometry of the pair $(C,0)\subset (X,0)$.
The goal is to recover $\delta(C)$ from the embedded topology of the pair under certain restrictions (eg. when
$(X,0)$ is rational).

One can attack this problem in two different ways, either analytically as in \cite{kappa1}, or via topological
tools, as in~\cite{kappa2}. The first one is based on vanishing theorems (valid from a certain bound of Chern classes),
the second one is based on the machinery of multivariable zeta (Poincar\'e) functions/series, their counting functions
and surgery formulae. These formulae are valid again only for certain bound of the corresponding Chern classes (in this case
one relies on asymptotically true formulae, and the bounds from which their validity is true has geometrical consequences).

In both theories, these bounds impose some restrictions on the topology of the curve $(C,0)$,
in such cases the discussion is optimal.
In this article we call such curves \emph{minimal generic curves} of $(X,0)$.

In the body of the paper we determine $\delta(C)$ for any `minimal generic curve germ' $(C,0)$ embedded into a
rational surface singularity $(X,0)$ from the topology of the embedding. The proof is based on a surgery formula of
the multivariable topological Poincar\'e series and several
recent developments in the area (see eg.~\cite{LNehrhart,LSzMonoids,LNN,LNNdual,kappa1,kappa2}).

Surprisingly, it turns out that the minimality condition imposed by being `minimal generic' is reflected as an extremal
behavior in the value of the delta invariant as well.

Recall that if $(C,0)$ is an abstract reduced curve germ with $r$ irreducible components, then $\delta(C)\geq r-1$,
and $\delta(C)=r-1$ is realized for the very special family of `ordinary $r$--tuples'. They are analytically isomorphic
with the union of the $r$ axes of $(\CC^r,0)$, or with the union of $r$ smooth generic lines in $(\CC^r,0)$.

It turns out that `minimal generic' curve germs on a cyclic quotient singularity are all ordinary $r$--tuples, a fact firstly noticed in \cite{CM} (a work, which partly motivated our work).
In the present note we indicate several proofs of this fact (of rather different nature).

But what is even more surprising is that this behavior almost survives whenever $(X,0)$ is an arbitrary
quotient singularity.

Our main result states that in such a case (i.e. if $(X,0)$ is quotient singularity and $(C,0)$ is a minimal
generic curve germ on it) $\delta(C)\in \{r-1,r\}$. ($\delta (C)=r$ for $r\geq 3$ means that $(C,0)$ is
isomorphic with $r$ lines in $(\CC^{r-1},0)$ in generic position.)

The paper is organized as follows.
In the preliminary section~\ref{sec:preliminaries} we recall all the needed material 
(with additional comments and remarks, in order to help
the understanding of the main connections and ideas) regarding normal surface singularities.
We discuss the definition of \emph{minimal generic curves}, their connection with special minimal Chern classes 
(rational divisors), and the related Laufer algorithm. Then we continue with a material related with multivariable 
series, the topological and analytical Poincar\'e series, counting functions, and finally the surgery formula. 
We refer to~\cite{LNehrhart,LSzMonoids,LNN,LNNdual,kappa1,kappa2} for more details.
Here we also discuss in short the connection with the kappa invariant, which was historically one of the major
motivations, cf.~\cite{ji-tesis,CM}.

In section~\ref{sec:cyclic} we discuss the case when $(X,0)$ is a cyclic quotient singularity. 
A formula for the delta invariant of minimal generic curves is already known (see~\cite{CM}), but we present
several alternative proofs using different techniques, which shed some light into the general problem.

Section~\ref{sec:delta-invariant} is devoted to developing the theory of minimal generic curves on rational 
surface singularities with star-shaped minimal resolution graphs (here in some parts the
rationality is not even needed). In Theorem~\ref{thm:SUM} we provide a formula for the delta invariant of 
minimal generic curves on rational surface singularities with a star-shaped graph.

In section~\ref{sec:quotient} the main result is applied to quotient surface singularities. As a result, 
in Theorem~\ref{thm:SUM2:alt} we describe the formula for the delta invariant of minimal generic curves
on quotient singularities based only on the information provided by the dual graph.

Section~\ref{ss:PROOF} is devoted to proving Theorem~\ref{thm:SUM2:alt} and describing the exceptional 
cases where the general formula does not apply. The different nature of these cases should definitely 
be better understood.

\section{Preliminaries}
\label{sec:preliminaries}

\subsection{Topological invariants of surface singularities}

\subsubsection{\bf Normal singularities, good resolutions, and dual graphs}
Let us consider a complex normal surface singularity $(X,0)$, that is, a complex surface germ such that its local ring
of functions $\cO_{X,0}$ is integrally closed in its field of fractions. Equivalently, this means that any bounded holomorphic
function on $X\setminus \{0\}$ can be holomorphically extended to~$X$.

Let $\pi:\tilde X\to X$ be a \emph{resolution}, that is, a proper holomorphic map from a smooth surface $\tilde X$ to a given
representative of $(X,0)$ such that $\pi$ is biholomorphic outside $\pi^{-1}(\{0\})$. Moreover, we will require $\pi$
to be a \emph{good resolution}, which means that the exceptional set $\pi^{-1}(\{0\})$ is a union of smooth curves that
intersect each other transversally.

A good resolution for normal surfaces always exists, but is not unique. Given a good resolution, we define its
\emph{dual graph}, as a decorated graph, say $\Gamma$, whose set of vertices, $\V$, is in bijection with
$\{E_v\}_{v\in \V}$ the irreducible components of the exceptional set of $\pi$. Two vertices $u,v$ are connected in $\Gamma$
if and only if $E_u\cap E_v\neq\emptyset$. Moreover, each vertex is decorated with two numbers, namely $E_v^2,g(E_v)$,
where $E_v^2$ is the self-intersection of the compact Riemann surface $E_v$ in $\tilde X$ and $g(E_v)$ is its genus.

Associated to the dual graph $\Gamma$ there is an \emph{intersection matrix} $M=(m_{u,v})_{u,v\in \V}$ given as the
intersection matrix of the curves $\{E_v\}_{v\in \V}$ in $\tilde X$, that is, $m_{u,v}:=(E_u, E_v)$. This matrix
is given as the numbers $E^2_v$ in the diagonal entries $m_{v,v}$ and the incidence matrix coefficients of $\Gamma$
outside the diagonal.

\subsubsection{\bf The link of a normal surface singularity}\label{sec:link}
If one assumes $(X,0)$ is embedded in an affine germ $(\CC^N,0)$, then the link $\Sigma$ of $(X,0)$ is defined as
the intersection of $X$ with a small enough $(2N-1)$-dimensional sphere centered in $0$.
The topology of $\Sigma$ does not depend on the embedding or the small-enough sphere.
Moreover, $(X,0)$ is a cone over~$\Sigma$ and
hence~$\Sigma\sim X\setminus \{0\}\cong \tilde X\setminus \pi^{-1}(0)$, $\Sigma \cong \partial \tilde X$,
where $\sim$ represents the same homotopy type.

In this paper we assume that {\it the link $\Sigma$ is a rational homology sphere} ($\QHS$ for short).
This is equivalent to asking the dual graph $\Gamma$ of a good resolution to be a connected tree,
where all $E_v$'s are rational, that is $g(E_v)=0$ for all~$v\in \V$.

\subsubsection{\bf Divisor lattice structure}\label{sec:lattice}
Define the lattice $L$ as $H_2(\tilde X,\ZZ)$. It is generated by the exceptional divisors $E_v$, $v\in \V$, that is,
$L=\oplus_{v\in \V} \ZZ\langle E_v \rangle$. In the homology exact sequence of the pair $(\tilde X, \Sigma)$ one has
$H_2(\Sigma,\ZZ)=0$, $H_1(\tilde X, \ZZ)=0$, hence the exact sequence has the form:
\begin{equation}
\label{eq:ses}
0 \to L \to H_2(\tilde X,\Sigma,\ZZ) \to H_1(\Sigma,\ZZ) \to 0.
\end{equation}
Set $L':= \Hom(H_2(\tilde X,\ZZ),\ZZ)$.
The Lefschetz-Poincar\'e duality $H_2(\tilde X,\Sigma,\ZZ)\cong H^2(\tilde X,\ZZ)$
defines a perfect pairing $L\otimes H_2(\tilde X,\Sigma,\ZZ)\to \ZZ$. Hence $L'$ can be identified with
$H_2(\tilde{X}, \Sigma, \ZZ)$. By~\eqref{eq:ses} $L'/L\cong H_1(\Sigma,\ZZ)$, which will be denoted by $H$.
Note that since $\Sigma$ is a $\QHS$, $H$ is a finite abelian group. In fact, even if $\Sigma$ is not
$\QHS$, since the intersection form on $L$ is non--degenerate,
$H_2(\tilde{X},\ZZ)\to H_2(\tilde{X},\Sigma, \ZZ)$ is injective, and $L'/L= {\rm Tors}(H_1(\Sigma,\ZZ))$.

As mentioned above, since the intersection form is non--degenerate, $L'$ embeds into
$L_{{\mathbb Q}}:=L\otimes {\mathbb Q}$, and it can be identified with the rational cycles
$\{\ell'\in L_{{\mathbb Q}} \mid (\ell',L)_{{\mathbb Q}}\in \ZZ\}$, where
$(\,,\,)$ denotes the intersection form on $L$ and $(\,,\,)_{{\mathbb Q}}$ its extension to $L_{{\mathbb Q}}$.
Hence, in the sequel we regard $L'$ as $\oplus_{v\in \V} \ZZ\langle E^*_v \rangle$,
the lattice generated by the rational cycles $E^*_v\in L_{{\QQ}}$,
$v\in \V$, where $(E_u^*,E_v)_{{\QQ}}=-\delta_{u,v}$ (Kronecker delta) for any $u,v\in \V$.
The inclusion $L\subset L'$ in the bases $\{E_v\}_v$ and $\{E_v^*\}_v$ is given by the intersection
matrix of $L$ as follows:
$$E_v=-\sum_{u\in \V} (E_v,E_u) E_u^*.$$
Given an element $\ell'\in L'$ we denote by $[\ell']\in H$ its class in $H$.
In fact, $L'\subset L_{{\mathbb Q}}$ sits naturally in the
superlattice $\frac{1}{d}L$, where $d:=|H|$ is the order of $H$
and $L$ is identified with its integral points.
The inclusion $L'\subset L_\QQ$ is given by $-M^{-1}$, where $M$ is the
intersection matrix of $\Gamma$. The matrix $-M^{-1}$ has positive entries, since $M$ is negative definite.
Note that the coefficients of the cycles $E^*_v$'s, are the columns of $-M^{-1}$.

The elements $E^*_v$ have the following geometrical interpretation as well: consider $\gamma_v\subset \tilde X$ a
curvette associated with $E_v$, that is, a smooth irreducible curve in $\tilde X$ intersecting $E_v$ transversally.
Then $\pi^*\pi_*(\gamma_v)=\gamma_v+E^*_v$ (see \ref{ss:TOTALTR} for the notation $\pi^*$).

\subsubsection{\bf The canonical divisor and canonical cycle}
Let $K_{\tilde X}$ be the canonical divisor in the smooth surface $\tilde X$. The canonical divisor in $X$ is defined as
$K_X:=\pi_*(K_{\tilde X})$. Note that $K_\pi:=K_{\tilde X}-\pi^*(K_X)$ has support on the exceptional set $\pi^{-1}(0)$.
The divisor $K_\pi$ is called the relative canonical divisor of $\pi$, and it is determined topologically by the
adjunction formula, which imposes the following linear system
\begin{equation}
\label{eq:KX}
(K_\pi+E_v,E_v)=-2, \textrm{ for all } v\in \V
\end{equation}
In particular $K_{\pi}\in L'$.
In some cases, it is more convenient to use the (anti)canonical cycle $\ZK:=-K_\pi$. Using~\eqref{eq:KX},
$\ZK$ can be written as
\begin{equation}\label{eq:ZK}
\ZK=E-\sum_{v\in \V} (2-{\rm val}_v) E^*_v,
\end{equation}
where $E=\sum_{v\in \V}E_v$ and ${\val}_v$ is the valency of $v$.

Recall that by the Riemann-Roch theorem $\chi(\cO_\ell)=-\frac{1}{2}(\ell, \ell-\ZK)$.
This motivates the definition $\chi(\ell'):=-\frac{1}{2}(\ell', \ell'-\ZK)\in\QQ$ for any $\ell'\in L'$.

\subsubsection{\bf The $H$--partition of $L'$}
The lattice $L'$ admits a partition parametrized by the group $H$, where for any $h\in H$ one sets
\begin{equation}
\label{eq:Lprime}
L'_h=\{\ell'\in L'\mid [\ell']=h\}\subset L'.
\end{equation}
Note that $L'_0=L$.
Given $h\in H$ one can define $r_h:=\sum_v r_v E_v\in L'_h$ as the unique element of $L_h'$ such that
$0\leq r_v<1$ for all $v$. Equivalently, $r_h=\sum_v \{l'_v\} E_v$ for any $\ell'=\sum_v l'_v E_v\in L'_h$,
where $0\leq \{\cdot\}<1$ represents the usual fractional part.

\subsubsection{\bf Examples. Quotient singularities}\label{sss:Q}
An infinite family of examples of normal surface singularities with a $\QHS$-link is provided by the quotient $X=\CC^2/G$
of $\CC^2$ by the action of a small finite subgroup $G$ of $\GL(2,\CC)$. These are called \emph{quotient surface} singularities.
 Note that $H=G/[G,G]$ is finite, hence any such $(X,0)$ has a $\QHS$-link.

If $G_1, G_2\subset \GL(2,\CC)$, then $X_1=\CC^2/G_1$ and $X_2=\CC^2/G_2$ are isomorphic if and only if $G_1$ and $G_2$
are conjugated. In particular, it can happen that $G_1$ and $G_2$ are isomorphic as abstract groups but $X_1$ and $X_2$ are not.

A full classification of the small finite subgroups of $\GL(2,\CC)$ up to conjugation was provided by
Brieskorn~\cite{Brieskorn-Rationale} (also a complete list of dual graphs is given by Nikulin in~\cite{Nikulin-delPezzo}).
The simplest ones are the finite cyclic groups, $G=\ZZ_d$. In this case any $0<q<d$, ${\rm gcd}(d,q)=1$ provides a
possible linear representations on $\CC^2$ given by the diagonal matrices
$$\xi\mapsto \left(
\array{cc}
\xi& 0\\ 0 & \xi^q
\endarray
\right),
$$
where $\ZZ_d=\{\xi\in \CC\,:\, \xi^d=1\}$. The quotient space usually is denoted by $X_{d,q}$, or $\frac{1}{d}(1,q)$.
The quotient
spaces $\frac{1}{d}(1,q_1)$ and $\frac{1}{d}(1,q_2)$ are isomorphic if and only if
$q_2\in\{q_1,q_1'\}$, where $q_1q_1'\equiv 1 \ ({\rm mod} \ d)$.

For $d=2$ there is only one quotient singularity, which can also be given by the germ $x^2+y^2+z^2=0$.
However, for $d=3$ there are two non--isomorphic ones: $\frac{1}{3}(1,1)$ is a complete intersection singularity
in $\CC^4$ whereas $\frac{1}{3}(1,2)$ is a hypersurface in $\CC^3$ (moreover, their embedding dimensions are 4 and
3 respectively).

For more details (including the minimal resolution graphs of $\frac{1}{d}(1,q)$) see section~\ref{sec:cyclic}.

From a different point of view, McKay~\cite{McKay-Graphs} proved that there is a 1--1 correspondence
between minimal graphs of quotient singularities and conjugacy classes of finite subgroups in $\GL(2,\CC)$.
This is known as the McKay correspondence.
(E.g., regarding the previous pair,
the graph of $\frac{1}{3}(1,1)$ has one vertex, whereas $\frac{1}{3}(1,2)$ has two vertices.)

\subsection{The notion of minimal generic curve germs of $(X,0)$}\label{ss:HrepLip}\
We define the \emph{rational Lipman cone} by
$$\cS_\QQ:=\{\ell'\in L_\QQ \mid (\ell',E_v)\leq 0 \ \mbox{for all} \ v\in \V\},$$
which is a cone generated over $\QQ_{\geq 0}$ by $E^*_v$.
By \ref{sec:lattice}, if $s\in \cS_\QQ\setminus \{0\}$ then all the $E_v$--coordinates of $s$ are strict positive.

Define $\cS':=\cS_\QQ\cap L'$ as the semigroup (monoid) of anti-nef rational cycles of $L'$; it is generated over
$\mathbb{Z}_{\geq 0}$ by the cycles $E^*_v$. As mentioned in section~\ref{sec:lattice}, any element of $\cS'$ can
be obtained as the exceptional part of the pull-back of an effective divisor of~$X$.

Set also $\cS:=\cS_\QQ\cap L$.

\subsubsection{\bf Artin's fundamental cycle and other minimal cycles}
The Lipman cone $\cS'$ also admits a natural equivariant partition $\cS'_{h}=\cS'\cap L'_h$ indexed by $H$.

For $\ell'_1,\ell'_2\in L_\QQ$ with $\ell'_i=\sum_v l'_{iv}E_v$ ($i=\{1,2\}$)
one considers an order relation $\ell'_1\geq \ell'_2$ defined coordinatewise by $l'_{1v}\geq l'_{2v}$
for all $v\in\V$. In particular, $\ell'$ is an effective rational cycle if $\ell'\geq 0$.
We set also $\min\{\ell'_1,\ell'_2\}:= \sum_v\min\{l'_{1v},l'_{2v}\}E_v$ and
analogously $\min\{F\}$ for a finite subset $F\subset L_\QQ$.

First note that $\cS=\cS'_0$ has the following properties:
\begin{enumerate}
 \item if $Z=\sum n_vE_v\in \cS$ and $Z\not=0$, then $n_v>0$ for all $v\in \V$,
 \item if $Z_1,Z_2\in \cS$, then $Z_1+Z_2\in \cS$,
 \item if $Z_1,Z_2\in \cS$, then $\min \{Z_1,Z_2\}\in \cS$.
\end{enumerate}
As a consequence, one can define the unique minimal element $\Zmin$
of $\cS\setminus \{0\}$, called the \emph{Artin's fundamental cycle}
(\emph{minimal cycle} or \emph{numerical cycle}).
It satisfy $\Zmin\geq E$.

Similarly, for any $h\in H$, $\cS_h$ has the following properties:
\begin{enumerate}\label{prop:GenLauferAlg}
\item $s_1,s_2\in \cS'_{h}$ implies $s_2-s_1\in L$ and hence $\min\{s_1,s_2\}\in \cS'_h$,
\item\label{prop:sh}
for any $s\in L_{{\QQ}}$ the set $\{s'\in \cS'_h \mid s'\not\geq s\}$ is finite,
\item for any $h$ there exists a unique \textit{minimal cycle} $s_h:=\min \{\cS'_{h}\}$~(see~\ref{ss:GLA} below).
\end{enumerate}
Note that for $h=0$ one has $s_0= \min\{\cS'_0\}=0$, while $\Zmin=\min \{\cS'_0\setminus 0\}\geq E$.

In general, for an arbitrary graph it is very difficult to determine the cycles $s_h$.
For strings or star-shaped graphs see e.g.~\cite{NemOSZ}. These will be revised in this note as well.

\subsubsection{\bf Generalized Laufer's algorithm}\label{ss:GLA}
\cite[Lemma 7.4]{NemOSZ} For any $\ell'\in L'$ there exists a unique minimal element $s(\ell')$ of the set
$\{s\in \cS' \mid s-\ell'\in L_{\geq 0}\}$. It can be obtained by the following algorithm.
Set $x_0:=\ell'$. Then one constructs a {\it computation sequence} $\{x_i\}_i$ as follows.
If $x_i$ is already constructed and $x_i\not\in\cS'$ then there exits some $E_{v(i)}$ such that $(x_i,E_{v(i)})>0$.
Then take $x_{i+1}:=x_i+E_{v(i)}$.
Then the procedure after finitely many steps stops,
say at $x_t$, and necessarily $x_t=s(\ell')$.

Note that $s(r_h)=s_h$ and $r_h\leq s_h$, however, in general $r_h\neq s_h$.
(This fact does not contradict the minimality of $s_h$ in $\cS'_h$ since $r_h$ might not sit in $\cS'_h$.)

\subsubsection{}\label{sss:LauferCrit} {\bf Laufer's criterion for rationality. }
The classical algorithm of Laufer constructs a computation sequence from $E$ (or, from any $E_v$) to $\Zmin$.
According to \cite{Laufer-rational}, $(X,0)$ is rational if and only if along this computation sequence
$(x_i,E_{v(i)})=1$ always.
(In particular, $\chi(\Zmin)=1$ too, this identity is Artin's criterion for rationality.)

\subsubsection{\bf Minimal generic $h$--curves on a surface singularity}\label{ss:TOTALTR}
Consider a reduced curve (or an effective Weil divisor) $(C,0)$ in $(X,0)$.
Let $\tilde C$ denote its strict transform by~$\pi$.
Then we define a cycle $\l'_C\in L'$ in such a way that the divisor $\l'_C+{\tilde C}$ is numerically trivial:
$(\l'_C+{\tilde C},E_v)=0$ for all $v$. We write this fact in the language of divisors as $\pi^*(C)=\l'_C+{\tilde C}$.

\begin{dfn}
\label{dfn:generic} Fix some $h\in H$ and the minimal good resolution. We say that $(C,0)$ is
\begin{enumerate}
 \item an \emph{$h$--curve} if $[\l'_C]=h$;
 \item a \emph{minimal} $h$--curve if $\l'_C=s_h$;
 \item a \emph{minimal generic} $h$--curve if $\l'_C=s_h$ and all the components of $\tilde C$ are
smooth and do not intersect each other, and they intersect $E$ transversally.
\end{enumerate}
Also, a curve $(C,0)\subset (X,0)$ is called \emph{minimal generic} if it is a 
minimal generic $h$--curve for some~$h\in H$.
\end{dfn}

Note that minimal generic curves were introduced as \emph{generic} curves in~\cite{CM} and studied in the context
of cyclic quotient singularities, however we think the term \emph{minimal generic} is more appropriate.

From the definition it follows that for $h=0$ the minimal $h$--curve is the empty curve, what we might include or
we might eliminate (though when we compare this picture with some other classification statements usually it is
convenient to have the empty curve also included).

Note that for fixed $h\in H$, all minimal generic $h$--curves are topologically equivalent: the number of components
intersecting transversally $E_v$ is exactly $c_v$, where $s_h=\sum_vc_vE_v^*$.
Analytically (in ${\rm Pic}(\tilde X)$) these curves might be different. However, if $(X,0)$ is rational
(when ${\rm Pic}(\tilde X)=L'$), for fixed $h$, all the minimal generic $h$--curves are linearly equivalent.
In this case we will call such a curve `the' minimal generic curve associated with~$h$.

A minimal generic curve $(C,0)$ is not necessarily irreducible and if irreducible it is not necessarily smooth
(though $\tilde C$ is smooth) (see Examples~\ref{sec:exceptional16} and~\ref{sec:exceptional16}).
Note also that if we take an irreducible smooth $\tilde C$ intersecting $E_v$ transversally,
then $C=\pi(\tilde C)$ is a generic $[E^*_v]$--curve, but not necessarily minimal.
Indeed, eg., since the $E_8$ graph is unimodular, the only minimal generic curve is the empty one.

In general, as $s_h$ is hard to characterize, the position of the (strict transforms of the) minimal generic
curves is also hard to characterize.

However, if $(X,0)$ is a cyclic quotient singularity, then any irreducible transversal smooth curve $\tilde C$
(intersecting any $E_v$) represents its minimal generic $h$--class (a fact proved in~\cite{CM}, see also the
detailed discussion in section~\ref{sec:cyclic}).
This provides in the case of the minimal resolution of a cyclic quotient a bijection
between irreducible minimal generic curves, the vertices of the dual graph, and the full special $\cO_X$-modules 
in the spirit of the McKay correspondence.

\subsection{Multivariable series}
\label{ss:set}
Let $\ZZ[[L']]$ be the $\ZZ$-module consisting of the $\ZZ$-linear combinations of the monomials
$\mathbf{t}^{\ell'}:=\prod_{v\in \V}t_v^{l'_v}$, where $\ell'=\sum_{v}l'_v E_v\in L'$.
Note that it is a $\ZZ$-submodule of the formal power series in variables $t_v ^{1/d}, t_v^{-1/d}$, $v\in V$.

Consider a multivariable series $S(\tt)=\sum_{\ell'\in L'}a(\ell')\tt^{\ell'}\in \ZZ[[L']]$.
Let $\Supp S(\tt):=\{\ell'\in L' \mid a(\ell')\neq 0\}$ be the support of the
series and we assume the following finiteness condition: for any $x\in L'$
\begin{equation}\label{eq:finiteness}
\{\ell'\in \Supp S(\tt)\mid \ell'\not\geq x\} \ \ \mbox{is finite}.
\end{equation}

Throughout this paper we will use multivariable series in $\ZZ[[L']]$ as well as in $\ZZ[[L'_\I]]$
for any $\I\subset \V$, where $L'_\I={\rm pr}_I(L')$ is the projection of $L'$ via
${\rm pr}_I:L_{{\QQ}} \to \oplus_{v\in \I} \QQ \langle E_v\rangle$.
For example, if $S(\tt)\in \ZZ[[L']]$ then $S(\tt_{\I}):=S(\tt)|_{t_v=1,v\notin \I}$ is an element of
$\ZZ[[L'_\I]]$. In the sequel we use the notation $\ell'_I=\ell'|_I:= {\rm pr}_I(\ell')$ and
$\tt^{\ell'}_\I:=\tt^{\ell'}|_{t_v=1,v\notin \I}$ for any $\ell'\in L'$.
Each coefficient $a_\I(x)$ of $S(\tt_\I)$ is obtained as a summation of certain
coefficients $a(y)$ of $S(\tt)$, where $y$ runs over
$\{\ell'\in \Supp S(\tt)\mid \ell'|_I =x \}$
(this is a finite sum by~\eqref{eq:finiteness}). Moreover, $S(\tt_\I)$ satisfies a similar
finiteness property as~\eqref{eq:finiteness} in the variables~$\tt_\I$.

Any $S(\tt)\in \ZZ[[L']]$ decomposes in a unique way as $S(\tt)=\sum_h S_h(\tt)$,
where $S_h(\tt):=\sum_{[\ell']=h}a(\ell')\tt^{\ell'}$. The series $S_h(\tt)$ is called the $h$-part of $S(\tt)$.
Note that the $H$-decomposition of the series $S(\tt_{\I})$ is not well defined. That is, the restriction
$S_h(\tt)|_{t_v=1,v\notin \I}$ of the $h$-part $S_h(\tt)$ cannot be recovered from $S(\tt_\I)$ in general,
since the class of $\ell'$ cannot be recovered from $\ell'|_I$. Hence, the notation $S_h(\tt_\I)$
(defined as $(S_h(\tt))|_{t_v=1,v\notin \I}$) is not ambiguous, but requires certain caution.

\subsubsection{\bf The multivariable topological Poincar\'e series}
The \textit{multivariable topological Poincar\'e series} (cf.~\cite{CDG-Poincare,CDGZ-PSuniversalAC,Nem-CLB})
is the Taylor expansion $Z(\tt)=\sum_{\ell'} z(\ell')\tt^{\ell'} \in\ZZ[[L']]$
at the origin of the rational \textit{zeta function}
\begin{equation}\label{eq:1.1}
f(\tt)=
\prod_{v\in \V} (1-\tt^{E^*_v})^{{\rm val}_v-2}.
\end{equation}
One can consider its $H$-decomposition $Z(\mathbf{t})=\sum_{h\in H}Z_h(\mathbf{t})$.
The support of $Z(\tt)$ is in the Lipman cone $\cS'$.
Note that by~\ref{prop:GenLauferAlg}\eqref{prop:sh} for any $x\in L'$ the set $\{\ell'\in \cS'\mid \ell'\not\geq x\}$
is finite. Hence for any $h\in H$ one can consider the {\it counting function} $Q_h^{\Gamma}(\tt)=Q_h(\tt)$ of
$Z_h(\tt)$ defined as
\begin{equation}\label{eq:count1}
Q_h(\tt): L'\longrightarrow \ZZ, \ \ \ \
x\mapsto \sum_{\ell'\ngeq x, \ [\ell']=h} z(\ell').
\end{equation}
Furthermore, for any $h\in H$ and $I\subset \V$ one can also consider the {\it counting function} $Q_{h,I}^{\Gamma}(\tt)=Q_{h,I}(\tt)$ of
$Z_h(\tt_I)$ defined as
\begin{equation}\label{eq:count2}
Q_{h,I}(\tt): L'\longrightarrow \ZZ, \ \ \ \
x\mapsto \sum_{\ell'|_I\ngeq x|_I,\ [\ell']=h} z(\ell').
\end{equation}

\subsubsection{\bf Surgery formulas for the counting function $Q_h(\tt)$}
\label{sec:surgform}
Let us fix $I\subset \V$. The set of vertices $\V\setminus I$
determines the connected full subgraphs $\{\Gamma_k\}_k$, $\cup_k\Gamma_k=\Gamma\setminus I$.
For each $k$ we consider the inclusion operator $j_k:L(\Gamma_k)\to L(\Gamma)$,
$E_v(\Gamma_i)\mapsto E_v(\Gamma)$, identifying naturally the corresponding $E$--base elements in the two graphs. This
preserves the intersection forms. Let $j_{k}^*:L'(\Gamma)\to L'(\Gamma_k)$ be its dual, defined by
$j_{k}^*(E^*_{v}(\Gamma))=E^*_{v}(\Gamma_k)$ if $v\in\V(\Gamma_k)$, and $j_{k}^*(E^*_{v}(\Gamma))=0$ otherwise.
Then one has the projection formula
$(j^*_{k}(\ell'), \ell)_{\Gamma_k}=(\ell',j_{k}(\ell))_{\Gamma}$ for any $\ell'\in L'(\Gamma)$ and $\ell\in L(\Gamma_k)$.
This implies that
\begin{equation}\label{eq:proj\ZK}
j^*_{k}(\ZK)=\ZK(\Gamma_k),
\end{equation}
where $\ZK=Z_K(\Gamma) $, respectively $\ZK(\Gamma_k)$, are the (anti)canonical cycles of $\Gamma$ and $\Gamma_k$ respectively.
Note also that $j^*_{k}(E_v(\Gamma))=E_v(\Gamma_k)$ for any $v\in \V(\Gamma_k)$.

Furthermore, one has the following formula regarding the pull--back of the cycles~$s_h$.

\begin{lemma}[\cite{LNN,LSzPoincare}]\label{lem:projs_h}
For any $h\in H$ one has $j^*_{k}(s_{h})=s_{[j^*_{k}(s_{h})]}\in L'(\Gamma_k)$.
\end{lemma}

For any fixed $h\in H$ one has the following surgery formula for the counting functions:

\begin{thm}[{\cite[Theorem 3.2.2]{LNN},\cite[Corollary~3.8]{kappa2}}]
\label{thm:surgform}
For any $\ell'=\sum_va_vE^*_v$ with $a_v\gg 0$ and with the notation $[\ell']=h$ one has the identity
\begin{equation}\label{eq:surgform}
Q^\Gamma_{h}\,(\ell')=Q^\Gamma_{h,\I}\,(\ell')+\sum_k \
Q^{\Gamma_k}_{[j^*_k(\ell')]}(j^*_k(\ell')).
\end{equation}
Moreover, if $\Gamma $ is rational, then ~\eqref{eq:surgform}
holds on the entire cone~$\ZK+\cS'$, and
 $Q_{[\ZK]}(\ZK)=0$.
\end{thm}

\subsubsection{\bf The multivariable analytic Poincar\'e series}
Parallel to $Z(\tt)$ there is a series defined using the analytic structure of $(X,0)$ (and a fixed resolution).
Although in this note it will not be really used, for the convenience of the reader we include its definition here
(together with the CDGZ--identity), since this correspondence was a leading motivation in the work of the authors, and
initiated and resulted several of the statements, which are central in this discussion.
(Eg., Theorem \ref{thm:surgform} was motivated by an analytic surgery formula.)

We fix a good resolution $\pi$ of $X$. Consider $c:Y\to X$, the universal abelian covering of $(X,0)$, let
$\tilde Y$ be the normalized pull-back of $\pi$ and $c$, and denote by $\pi_Y$ and $\tilde{c}$
the induced maps by the pull-back completing the following commutative diagram.
\begin{equation}
\label{eq:diagram}
\xymatrix{
\ar @{} [dr] | {\#}
\tilde{Y} \ar[r]^{\tilde{c}} \ar[d]_{\pi_Y} & \tilde{X} \ar[d]^{\pi} \\
Y \ar[r]_{c} & X
}
\end{equation}

We define the following $H$-equivariant $L'$-indexed divisorial filtration of the local ring $\mathcal{O}_{Y,0}$:
for any given $\ell'\in L'$ set
\begin{equation}
\label{eq:Ffiltration}
\mathcal{F}(\ell'):=\{g\in \cO_{Y,0} \mid \div (g\circ \pi_Y)\geq \tilde c^* (\ell')\}.
\end{equation}
It is worth mentioning
that the pull-back $\tilde c^*(\ell')$ is an integral cycle in $\tilde Y$ for any $\ell'\in L'$,
cf.~\cite[Lemma 3.3]{Nem-CLB}. The natural action of $H$ on $(Y,0)$ induces an action on $\cO_{Y,0}$ as follows:
$h\cdot g(y)=g(h\cdot y)$, $g\in \cO_{Y,0}$, $h\in H$. This action decomposes
$\cO_{Y,0}$ as $\oplus_{\lambda\in \hat{H}} (\cO_{Y,0})_{\lambda}$ according to the characters
$\lambda \in \hat{H}:={\rm Hom}(H,\CC^*)$, where
\begin{equation}
\label{eq:Heigenspaces}
(\cO_{Y,0})_{\lambda}:=\{g\in \cO_{Y,0} \mid g(h\cdot y)=\lambda (h)g(y),\ \forall y\in Y, h\in H\}.
\end{equation}
Note that there exists a natural isomorphism $\theta:H\to \hat{H}$ given by
$h\mapsto \exp(2\pi \sqrt{-1} (\ell',\cdot ))\in \Hom(H,\CC^*)$, where $\ell' $ is any element of $L'$ with
$h=[\ell']$.

The subspace $\mathcal{F}(\ell')$ is invariant under this
action and $ \mathcal{F}(\ell')_{\theta(h)}=\mathcal{F}(\ell')\cap (\cO_{Y,0})_{\theta(h)}$. Thus, one can define the \textit{Hilbert function}
 $\mathfrak{h}(\ell')$ for any $\ell'\in L'$
as the dimension of the $\theta([\ell'])$-eigenspace $(\cO_{Y,0}/\mathcal{F}(\ell'))_{\theta([\ell'])}$.
The corresponding multivariable \textit{Hilbert series} is
\begin{equation}
\label{eq:Hilbert}
H(\mathbf{t})=\sum_{\ell'\in L'} \mathfrak{h}(\ell')\mathbf{t}^{\ell'}\in \mathbb{Z}[[L']].
\end{equation}

The \textit{multivariable analytic Poincar\'e series}
$P(\mathbf{t})=\sum_{\ell'\in L'}\mathfrak{p}(\ell')\mathbf{t}^{\ell'}$ can be defined by
\begin{equation}
\label{eq:analPoincare}
P(\mathbf{t})=-H(\mathbf{t})\cdot \prod_{v\in \V}(1-t_v^{-1}).
\end{equation}
\subsubsection{\bf The CDGZ-identity}
\label{sec:pg}
The \textit{CDGZ-identity}, named after Campillo, Delgado and Gusein-Zade thanks to their contributions
in~\cite{CDG-Poincare} and~\cite{CDGZ-PSuniversalAC} relates the Poincar\'e series $P(\tt)$ and $Z(\tt)$ for
rational surface singularities.

\begin{thm}[CDGZ-identity \cite{CDG-Poincare}]\label{thm:CDGZ}
A rational surface singularity $(X,0)$ and any of its resolution $\pi$ satisfy the identity
\begin{equation}\label{CDGZ}
P(\tt)=Z(\tt).
\end{equation}
\end{thm}
The fourth author in~\cite{Nem-PS},~\cite{Nem-CLB} extended the result for larger
families of normal surface singularities with rational homology sphere link.
The largest family for which the CDGZ-identity holds is the family of
splice-quotient singularities with their special resolutions satisfying the end-curve conditions
(see~\cite{Nem-ICM}).

\subsection{Invariants of abstract curve germs}\label{ss:AbCu}
In this section we recall the most relevant local invariants of a curve germ considered in this paper.

\subsubsection{\bf The delta invariant.}
The delta invariant $\delta(C)$ of a reduced curve germ $(C,0)$ is defined as follows.
Let $\gamma:(C,0)^{\widetilde{}}\to(C,0)$
be the normalization, where $(C,0)^{\widetilde{}}$ is the corresponding multigerm. 
Then $\delta(C):=\dim_{{\CC}}\gamma_*\cO_{(C,0)^{\widetilde{}}}/\cO_{(C,0)}$.

Usually we also write $r$ (or $r(C)$) for the number of irreducible components of $(C,0)$.

The delta invariant of a reduced curve can be determined inductively from the delta invariant of
the components and the \emph{Hironaka generalized intersection multiplicity}.
Indeed, assume $(C,0)$ is embedded in some $(\CC^n,0)$, and assume that $(C,0)$ is the union of
two (not necessarily irreducible) germs $(C_1',0)$ and $(C_2',0)$ without common irreducible components.
Assume that $(C_i',0)$ is defined by the ideal $I_i$ in $\cO_{(\CC^n,0)}$ ($i=1,2$).
Then one can define \emph{Hironaka's intersection multiplicity} by
$(C_1',C_2')_{Hir}:=\dim _{\CC}\, (\cO_{(\CC^n,0)}/ (I_1+I_2))$.
Then, one has the following formula of Hironaka, see~\cite{Hironaka} or~\cite[2.1]{StevensThesis} and~\cite{BG80},
\begin{equation}\label{eq:deltaA1}
\delta(C)=\delta(C_1')+\delta(C_2')+(C_1',C_2')_{Hir}.
\end{equation}
In particular, if $(C,0)$ has $r$ irreducible components, then using induction
 (and the obvious inequality $(C_1',C_2')_{Hir}\geq 1$) we get
\begin{equation}\label{eq:deltaA2}
\delta(C)\geq r-1.\end{equation}

\subsubsection{\bf Ordinary $r$--tuples} \label{ex:delta22} \cite{BG80,Greuel,Greuel2,StevensThesis}
If a reduced (abstract) curve germ $(C,0)$ is (analytically equivalent with) the union of the coordinate
axes of $(\CC^r,0)$, then we call $(C,0)$ an {\it ordinary $r$--tuple}. Using e.g.
Hironaka's formula we get that $\delta(C)=r-1$. This shows that the ordinary $r$--tuples realizes
the optimal minimum of the bound~\eqref{eq:deltaA2}. In fact, this property characterizes them:
if for a curve $(C,0)$ with $r$ components one has $\delta(C)= r-1$
then $(C,0)$ is necessarily an ordinary $r$--tuple. Their notation is $R^r_r$.

\subsubsection{\bf $R^{r-1}_r$ type $r$--tuples} \label{ex:delta33} \cite{BG80,Greuel,Greuel2,StevensThesis}
Recall that a curve of type $R^r_r$ consists of $r$ lines in $(\CC^r,0)$ in general position.
We generalize this as follows. For $r\geq 3$, $R^{r-1}_r$ denotes (the isomorphism class of) $r$ lines in
$(\CC^{r-1},0)$ in general position. The plane curve singularity $A_3$ (resp. $A_2$)
with equation $\{x^4+y^2=0\}$ (resp. $\{x^3+y^2=0\}$) will also be denoted by $R^1_2$ and $R^0_1$ respectively.
Then, it turns out that a curve singularity $(C,0)$ with $r$ components is $R^{r-1}_r$ if and only if $\delta(C)=r$.

\subsection{Invariants of embedded curve germs}\label{ss:EmbCu}

We fix a normal surface singularity $(X,0)$ with a $\QHS$-link, and we consider a reduced curve germ $(C,0)\subset (X,0)$.
Next we recall the definition and some properties of the {\it kappa invariant} $\kappa_{X}(C)$,
an embedded invariant of the pair $(C,0)\subset (X,0)$.
The interplay between $\delta(C)$ and $\kappa_{X}(C)$
has been studied in~\cite{kappa1,kappa2}. This connection produced several explicit formulae
for the delta invariant of minimal generic curves embedded into rational surface singularities.
Furthermore, the definition (and several expressions) of $\kappa_{X}(C)$ explains and motivates
several expressions intensively used in the present note.

\subsubsection{\bf The \texorpdfstring{$\kappa$}{kappa}--invariant of a cycle} Fix a good resolution
$\pi :\tX\to X$. For any $\ell'\in L'$ define
\begin{equation}\label{kappa_big}
\kappa_{\tX}(\ell'):=
\mathfrak{h}(Z_K+\ell')=\dim_\CC \Big(\frac{\cO_{Y,0}}{\mathcal{F}(Z_K+\ell')}\Big)_{\theta([Z_K+\ell'])},
\end{equation}
where $(Y,0)$ is the universal abelian covering of $(X,0)$ as described in~\eqref{eq:diagram}.
Note that $\kappa_{\tX}(\l')$ in this way is defined via the analytic filtration $\l'\mapsto {\mathcal F}(\l')$,
or equivalently, in terms of the Hilbert series $H(\tt)$, which can be identified with the counting function of the
analytic Poincar\'e series $P(\tt)$. However, if $(X,0)$ is rational, then by the CDGZ--identity~\ref{thm:CDGZ}
$P(\tt)=Z(\tt)$, hence $\kappa_{\tX}(\l')$ is computable from the counting function of $Z(\tt)$ as
$Q_{[Z_K+\l']}(Z_K+\l')$, cf.~\cite{kappa1,kappa2}. (In general, this expression is called the `topological' $\kappa$--invariant.)

Next, assume that $(C,0)$ is a reduced curve germ on $(X,0)$ (with $\QHS$ link, but not necessarily rational)
and choose a good embedded resolution $\pi: \tX\to X$ of the pair $(C,0)\subset (X,0)$. Then consider the strict transform
$\widetilde{C}$ and the total transform $\pi^*(C)=\ell_C'+\widetilde{C}$ as in~\ref{ss:TOTALTR}. Note that
$\ell'_C\in \mathcal{S}'$. It turns out that the invariant $\kappa_{\tX}(\ell'_C)$ is independent
of the resolution $\pi$ ~(see~\cite[Corollary 3.4]{kappa1}).

This justifies the following definition of the $\kappa$--invariant of a reduced curve germ $(C,0)\subset(X,0)$,
where $(X,0)$ has a $\QHS$ link.

\begin{dfn}
\label{def:kappa}
The $\kappa$--invariant of the pair $(C,0)\subset (X,0)$ is defined as
\begin{equation}
\kappa_X(C)=\kappa_{\tX}(\ell'_C)=\dim_\CC \Big(\frac{\cO_{Y}}{\mathcal{F}(Z_K+\ell'_C)}\Big)_{\theta([\ZK+\ell'_C])}.
\end{equation}
\end{dfn}
One of the main results of~\cite{kappa1} relates $\kappa_X(C)$ and the $\delta$--invariant of $(C,0)$ as follows.

\begin{thm}[{\cite[Theorem 1.1 \& 1.2]{kappa1}}]
\label{thm:mainkappa}
If $(X,0)$ is a rational surface singularity and $(C,0)\subset (X,0)$ a reduced curve germ, then
\begin{equation}
\label{eq:mainkappa}
\kappa_{X}(C)=\delta(C)=\chi(-\ell_C')-\chi(s_{-[\ell'_C]}).
\end{equation}
\end{thm}

In this paper we consider the problem of calculating the delta invariant of minimal generic curves
embedded into rational surface singularities.

\section{Old and new results for cyclic quotient singularities }\label{sec:cyclic}

\subsection{}
Let $(X,0)$ be a cyclic quotient surface singularity with the notation $(X,0):=\frac{1}{d}(1,q)$, cf. \ref{sss:Q}.
The dual graph $\Gamma$ of the minimal resolution of $X$ is the following string of vertices (see eg. \cite{BPV-book})

\begin{center}
\begin{tikzpicture}
\coordinate (A) at (-3,0);
\draw[fill] (A) circle (0.09);
\coordinate (B) at (-2,0);
\draw[fill] (B) circle (0.09);
\coordinate (C) at (-1,0);
\draw[fill] (C) circle (0.09);
\path [-, thick] (A) edge (B);
\path [-, thick] (B) edge (C);
\coordinate (D) at (-0.5,0);
\path [-, thick] (C) edge (D);
\coordinate (E) at (0.3,0);
\draw [dotted, thick] (-0.3,0)--(0.1,0);
\coordinate (F) at (0.8,0);
\path [-, thick] (E) edge (F);
\draw[fill] (F) circle (0.09);
\node[above=4] at (A) {$-k_1$};
\node[above=4] at (B) {$-k_2$};
\node[above=4] at (C) {$-k_3$};
\node[above=4] at (F) {$-k_s$};
\node[below=4] at (A) {\small $E_1$};
\node[below=4] at (B) {\small $E_2$};
\node[below=4] at (C) {\small $E_3$};
\node[below=4] at (F) {\small $E_s$};
\end{tikzpicture}
\end{center}
where the numerical data $\{k_i\}_{i=1}^s$ ($k_i\geq 2$) is encoded by the Hirzebruch--Jung (negative) continued fraction expansion
$$\frac{d}{q}=k_1-\frac{1}{k_2-\frac{1}{k_3-\dots}}:=[k_1,\dots,k_s].$$
We also set $0<q'<d$ so that $qq'\equiv 1$ (mod $d$). Moreover, for any $1 \leq i \leq j \leq s$ we write the continued fraction
$[k_i ,\dots, k_j]$ as a rational
number $d_{ij}/q_{ij}$ with $d_{ij}>0$ and $\mathrm{gcd}(d_{ij},q_{ij})=1$. Notice that $d_{ij}$ is the determinant
of the intersection matrix associated with the substring with vertices $v_i,\dots,v_j$. Then, $d_{1,s}=d$, $d_{2,s}=q$ and $d_{1,s-1}=q'$. We also set $d_{i,i-1}:=1$ and
$d_{ij}:=0$ for $j< i-1$.

Below we use the notation $\lfloor - \rfloor$
for the integral part, $\{-\}$ for the fractional part, and $\lceil -\rceil$ for the ceiling function.

We choose the class $[E^*_s]$ as a generator of $H=\mathbb{Z}_d$.
For any class $a[E^*_s]$ with $0\leq a< d$ we consider the minimal cycle $s_{a[E^*_s]}$ and we write
$s_{a[E^*_s]}=\sum_{i=1}^s a_i E^*_i$.
Then by~\cite[section 10.3]{NemOSZ}, one has
\begin{equation}\label{eq:afromai}
a=d_{2,s}\cdot a_1+d_{3,s}\cdot a_2+\ldots+d_{s,s}\cdot a_{s-1}+a_s,
\end{equation}
and the coefficients $a_i$ are expressed from the integer $a$ via the following recursive identity
\begin{equation}\label{eq:cyclicRid}
 a_i=\Big\lfloor \frac{a-\sum_{j=1}^{i-1} d_{j+1,s}\cdot a_j}{d_{i+1,s}} \Big\rfloor \ \ (1\leq i\leq s);
\end{equation}
see also \cite[Section 2.2]{CM}. In particular, $a_1=\lfloor a/q \rfloor$.

For each fixed $h=[aE^*_s]$ the entries $\{a_i\}_{i=1}^s$ identify the minimal generic $h$--curve.

\subsubsection {\bf Example.} Assume that $n/q=15/11=[2,2,3,2,2]$. Then the possible
entries $(a_1,\ldots, a_5)$ representing some $s_{a[E^*_s]}$ are the following $5$--tuples
$$(1,0,1,0,0), \ (1,0,0,1,0), \ (1,0,0,0,1), \ (1,0,0,0,0), \ (0,1,1,0,0), $$
$$(0,1,0,1,0), \ (0,1, 0,0,1), \ (0,1,0,0,0), \ (0,0,2,0,0), \ (0,0,1,1,0), $$
$$(0,0, 1,0,1), \ (0,0,1,0,0),\ (0,0,0,1,0), \ (0,0,0,0,1), \ (0,0,0,0,0). $$

\subsection{}\label{sec:otherproofs} We consider a minimal generic curve $(C,0) $ on $(X,0)$ whose associated cycle is $\ell'_C=\sum_{i=1}^s a_iE^*_i=s_h$ for some $h=
a[E^*_s]$, $0<a<s$.
If we denote by $r(C)$ the number of irreducible components of $C$, then we get $r(C)=\sum_i a_i$.

In this case one has the following simple formula for the delta invariant of $(C,0) $ in terms of the number of its branches.

\begin{thm}[\cite{CM}]\label{thm:cyclic}
 For a minimal generic curve $C$ on a cyclic quotient singularity $(X,0)$ one has $\delta(C)=r(C)-1$. In particular, $(C,0)$
is an ordinary $r(C)$--tuple. (Cf. \ref{ex:delta22}.)
\end{thm}

In the rest of the section, we summarize the existing proofs for the formula of Theorem~\ref{thm:cyclic}.
\vspace{0.3cm}

\subsubsection{} \label{sec:proof1}
The {\bf first proof} was developed by the first and third authors in~\cite{CM} using the methods of weighted
blow-ups and an analytic type surgery formula for $\kappa_X(C)$. The authors used explicit equations for the
irreducible minimal generic curves as $x^{q_i}-y^{\bar q_i}$, where $q_i$ is a remainder appearing in the Hirzebruch-Jung
decomposition of $\frac{d}{q}$ and $\bar q_iq=q_i \mod d$. The key of the proof is to show that $\kappa_X(C)$ can
be obtained recursively as $\kappa_X(C)=\kappa_\pi + \kappa_{X'}(C')$ where $\pi':X'\to X$ is a weighted blow-up
and $\kappa_\pi$ is either $k_i$ or $k_i-1$, where $k_i$ is the number of transversal cuts on the exceptional divisor,
depending on whether or not $\pi$ is the last blow-up in the resolution. The invariant $\kappa_\pi$ is obtained
as the number of integer points in a triangle
\begin{equation}
\label{eq:kappa-pi}
\kappa_{\pi}:= \# \{ (i,j) \in \mathbb{Z}^2 \mid i,j \geq 1, \, k \geq i +qj \equiv k \! \mod d \}=
\begin{cases}
k_i-1 & \text{ if } k=k_iq\\
k_i & \text{ otherwise}.
\end{cases}
\end{equation}

\subsubsection{}\label{sec:proof2}
The {\bf second proof} follows from one of the main result of~\cite{kappa1} and~\cite{kappa2} saying that
for rational singularities $(X,0)$ and for $(C,0)$ as above one has $\delta(C)=\chi(-s_h)-\chi(s_{-h})$.
The point is that when $(X,0) $ is cyclic quotient, the term $Exp:=\chi(-s_h)-\chi(s_{-h})$ can be calculated explicitly.
We follow~\cite[Example 5.9]{kappa1}. Indeed, we may write
\begin{equation}\label{eq:Exp1}
Exp=\chi(s_h)-\chi(s_{-h})-(s_h,\ZK)
\end{equation}
since $\chi(-s_h)=\chi(s_h)-(s_h,\ZK)$.
Then one can calculate
$$\chi(s_h)=\chi(s_{[aE^*_s]})=\frac{a(1-d)}{2d} + \sum_{i=1}^a\, \left\{ \frac{iq'}{d}\right\}$$
using ~\cite[10.5.1]{NemOSZ}.
Since $-h=[(d-a)E^*_s]$, $\chi(s_{-h})$ equals
\begin{equation}\label{eq:diff}
\begin{aligned}
 \chi(s_{-h}) &=
\frac{(d-a)(1-d)}{2d} + \sum_{j=1}^{d-a}\, \left\{ \frac{jq'}{d}\right\}\\
& = \frac{(d-a)(1-d)}{2d} + \sum_{i=1}^{d-1}\,\left( 1- \left\{ \frac{iq'}{d}\right\}\right)
 -\sum_{i=1}^{a-1}\,\left( 1- \left\{ \frac{iq'}{d}\right\}\right),
\end{aligned}
\end{equation}
which implies
$$\chi(s_h)-\chi(s_{-h})=\frac{a}{d}-1 -\left\{\frac{aq'}{d}\right\}.$$
On the other hand $\ZK=E-E_1^*-E_s^*$ by~\eqref{eq:ZK}, hence we get $(s_h,\ZK)=(s_h,E)-(s_h,E_1^*+E_s^*)$
where the first term is easily identified with $(s_h,E)=-r(C)$.
For calculating the second term first we write $-(E_1^*, E_i^*)=d_{i+1,s}/d$ ($1\leq i\leq s$) following \cite{NemOSZ}.
Then $(s_h,E_1^*)=-(\sum_i d_{i+1,s}\cdot a_i)/d=-a/d$ follows from~\eqref{eq:afromai}.
In order to compute $(s_h,E_s^*)$ we create formally the symmetric situation.
Since $[E^*_1]=[qE^*_s]$ and $[E^*_s]=[q'E^*_1]$,
$h=[aE^*_s]=[aq'E_1^*]$. Write $a' \equiv aq'$ (mod $d$), $0<a'< d$. Then for $h=[a'E_1^*]$ and
$s_h=s_{[a'E_1^*]}=\sum_i a_iE^*_i$ the symmetric formula is
\begin{equation*}(s_h,E_s^*)=-\frac{a'}{d}=\left\lfloor
 \frac{aq'}{d}\right\rfloor - \frac{aq'}{d}.\end{equation*}
Thus, by above calculation we get $(s_h,\ZK)=a/d-\lfloor aq'/d\rfloor+ aq'/d-r(C)$.
Finally, using also~\eqref{eq:Exp1} and~\eqref{eq:diff} we deduce~$Exp=r(C)-1$.

\subsubsection{}\label{sec:proof3}
The {\bf third proof}
seems to reinterpret~\ref{sec:proof1} in the language of Poincar\'e series.
It is based on the surgery formula~\eqref{eq:surgform}.
It is more in the spirit of the topological approach developed in~\cite{kappa2}. Here we use surgery
formula techniques for the counting function of the coefficients of the topological Poincar\'e series,
which will also serve as a base for further generalizations.
\vspace{0.3cm}

We proceed by induction on the number of vertices $|\V|$ of $\Gamma$. First we assume that $\Gamma$ contains
only one vertex with self-intersection $-d$. Then $\ZK=(1-2/d)E$ and for any $0< a< d$ the cycle of the generic
curve is $\ell'_C=s_{[aE^*]}=aE^*=(a/d)E$ (cf.~\eqref{eq:cyclicRid}). The topological Poincar\'e series associated
with $\Gamma$ is expressed as $Z(t)=\sum_{x\geq 0} (x+1) t^{x/d}$. By Theorem~\ref{thm:mainkappa} we know that
$\delta(C)=Q_{[\ZK+\l'_C]}\,(\ZK+\l'_C)$. In order to calculate it, we
need to sum all the coefficients $(x+1)$ corresponding to $x\geq 0$, for which
$x< d-2+a$ and $x\equiv d-2+a \ (\mathrm{mod} \ d)$. Since there exists only one such $x$, namely $x=a-2$, we get
that $\delta(C)=a-1$, which proves the statement for~$|\V|=1$.

Now assume $|\V|>1$. Let $v_1$ be the first vertex of the string and we denote by $\Gamma_2$ the substring which
we get from $\Gamma$ by deleting $v_1$ and its adjacent edge. Denote the inclusion operator by
$j_2:L(\Gamma_2)\to L(\Gamma)$ and consider its dual operator $j^*_2$ as in section~\ref{sec:surgform}.
Then we apply the surgery formula in Theorem~\ref{thm:surgform} for $v_1$ and the cycle $\ZK+\ell'_C\in \ZK+\mathcal{S}'$
in the rational surface context. Thus, we have
\begin{equation}\label{eq:surgcyc}
Q_{[\ZK+\ell'_C]}\,(\ZK+\ell'_C)=Q_{[\ZK+\ell'_C],v_1}\,(\ZK+\ell'_C)+
Q^{\Gamma_2}_{[j^*_2(\ZK+\ell'_C)]}(j^*_2(\ZK+\ell'_C)).
\end{equation}
First of all notice that by~\eqref{eq:projZ_K} one has $j^*_2(\ZK)=Z_{K}(\Gamma_2)$ where $Z_{K}(\Gamma_2)$ is the anti-canonical
cycle associated with $\Gamma_2$. We set the notation $\ell'_{C,2}:=j^*_2(\ell'_C)=\sum_{i=2}^s a_i E^*_i$ for the
pull--back of the cycle $\ell'_C=\sum_{i=1}^s a_i E^*_i$.

Since $C$ is a minimal generic curve one has $\ell'_C=s_h$ for $h=[\ell'_C]\in H$, hence by Lemma~\ref{lem:projs_h} we obtain
that $\ell'_{C,2}$ is the cycle of a minimal generic curve $C_2$ on the cyclic quotient singularity whose minimal resolution
graph is $\Gamma_2$.
Note that $\ell'_{C,2}=0$ exactly when $\sum_{i=2}^sa_i=0$, in this case the corresponding restricted curve is empty.
From~\eqref{eq:afromai} this happens exactly when $q|a$.
Thus, by the inductive step, the second term of the right-hand side of~\eqref{eq:surgcyc} is
\begin{equation}\label{eq:INDIND}Q^{\Gamma_2}_{[Z_{K,2}+\ell'_{C,2}]}(Z_{K,2}+\ell'_{C,2})=\delta(C_2)=\left\{
\begin{array}{ll} \textstyle{-1+\sum}_{i=2}^s a_i & \mbox{if $q\nmid a$} \\
0 & \mbox{otherwise}.\end{array}\right. \end{equation}
It remains to calculate $Q_{[\ZK+s_h],v_1}\,(\ZK+s_h)$ from the topological Poincar\'e series $Z(\tt)$ of $\Gamma$.
$Z(\tt)$, by definition, can be written in the following form
$$Z(\mathbf{t})=T\Big[\frac{1}{(1-\mathbf{t}^{E^*_1})(1-\mathbf{t}^{E^*_s})}\Big]=
\sum_{x,y\geq 0}\mathbf{t}^{x E^*_1+y E^*_s}.$$
Then we have to count all the cycles $\ell'=x E^*_1+y E^*_s$ satisfying the following properties:
\begin{align*}
 (i) \ \ \ & \ell'|_{E_1}<(\ZK+s_h)|_{E_1}\\
 (ii) \ \ \ & [\ell']=[\ZK+s_h].
\end{align*}

Recall that $h=[aE^*_s]$ for some $0<a<d$. Then one has $s_h|_{E_1}=a/d$. Similarly $[\ZK]=[(d-q-1)E^*_s]$,
hence $\ZK|_{E_1}=(d-q-1)/d$. Moreover, since $E^*_1|_{E_1}=q/d$ and $E^*_s|_{E_1}=1/d$ the above counting problem
with conditions ({\it i}) and ({\it ii}) can be transformed into the following lattice points counting:
we count pairs of integers $x', y'\geq 1$ satisfying
$$(i)\ \ x'q+y'< d+a\ \ \ \mbox{and} \ \ \ (ii) \ \ x'q+y'\equiv a \ (\mathrm{mod} \ d).$$
This means that $x'q+y'=a$ with $x'\geq 1$ and $y'\geq 1$.
 The number of such pairs
$(x',y')$ is exactly $\lfloor a/q\rfloor$ if $q\nmid a$, and it is $\lfloor a/q\rfloor-1 $ if $q| a$.
Finally note that from ~\eqref{eq:cyclicRid} $\lfloor a/q\rfloor=a_1$. Hence these facts combined
with~\eqref{eq:INDIND} ends the proof.

The connection between the counting function $Q_{[\ZK]+h,v}$ and the invariant $\kappa_\pi$ from the first
proof~\eqref{eq:kappa-pi} remains to be better understood.

\subsubsection{}\label{sec:proof4}
The {\bf fourth proof}, again {\bf new}, appears in~\cite{NEW}.
Originally we planned to insert it in~\cite{kappa1}, or as an Appendix of the present manuscript,
but during the reduction it evolved into an independent manuscript.
It considers the canonical (minimal) embedding of the cyclic quotient singularity into $\CC^e$ via its invariant monomials,
and then one studies the tangent application of the parametrized curves $\widetilde {C}\to \CC^e$.
One of the advantages of this proof is that it can be naturally generalized for more general
classes of singularities $(X,0)$, and
provides a direct method to compute the delta invariant of $(C,0)$ even in those cases when
the statement of Theorem \ref{thm:cyclic} is not true.

\begin{rem}
The proofs in \S\ref{sec:proof1} and \S\ref{sec:proof3} end up with a certain lattice point counting.
In fact, in proof in \S\ref{sec:proof2} this aspect appears subtly in the arithmetical computation as well:
behind these arithmetical formulae one can recognize either lattice point counting, or the corresponding
Dedekind sum interpretations. The third proof makes connection with surgery formulae as well via
explicit calculations of Ehrhart type quasi-polynomials. The proof \S\ref{sec:proof4} is completely different,
it is more algebraic.
\end{rem}

\section{Delta invariant of minimal generic curves on rational singularities} 
\label{sec:delta-invariant}

\subsection{The surgery formula}\label{sec:surgformtopkappa}
Let $\Gamma$ be the minimal good dual resolution graph of a rational singularity $(X,0)$.
Consider a minimal generic curve $(C,0)$ on $(X,0)$ with its associated cycle $s_h$ for some $h\in H$, $h\not=0$.

Let $\mathcal{N}$ denote the set of nodes of $\Gamma$, i.e. vertices with valency $\kappa_v\geq 3$.
Assume that $\mathcal{N}\neq \emptyset$ (the $\mathcal{N}= \emptyset$ case has been already discussed in
section~\ref{sec:cyclic}). Then the other vertices $\V\setminus \mathcal{N}$ determine the collection of connected strings
$\{\Gamma_k\}_k$, which are minimal resolution graphs of cyclic quotient singularities. We also consider for
any $\Gamma_k$ the inclusion operator $j_k$ and its dual $j^*_k$ considered in section~\ref{sec:surgform}.

Fix $h\in H$, $h\not=0$, and write
$s_h=\sum_{v\in\V} a_v E^*_v=\sum_{v\in\mathcal{N}}a_v E^*_v+\sum_k \sum_{v\in\V(\Gamma_k)}a_{v,k} E^*_{v,k}$
with $a_v\in\mathbb{Z}_{\geq 0}$. Then for every $k$ we have the cycles
$s_{h,k}:=j^*_k(s_h)=\sum_{v\in\V(\Gamma_k)}a_{v,k} E^*_{v,k}$ on the minimal resolution of the corresponding
cyclic quotient singularities $(X_k,0)$. Each $s_{h,k}$ is also an irreducible minimal cycle 
$s_{h,k}=s_{[j^*_k(s_h)]}$ on $\Gamma_k$
by Lemma~\ref{lem:projs_h}. It corresponds to the curve $(C_k,0)$, the union of those components of
$(C,0)$, which intersect $E(\Gamma_k)=\cup_{v\in \V(\Gamma_k)}E_v$. The curve $(C_k,0)$ is nonempty if and only if
$r_k(C):=\sum_{v\in\V(\Gamma_k)}a_{v,k} >0$. In such a case $r_k(C)$ is the
 the number of branches of $(C_k,0)$. Hence, whenever $s_{h,k}\neq 0$ (i.e. $C_k$ is nonempty),
by Theorem \ref{thm:cyclic} one obtains that
$\delta(C_k)= \sum_{v\in\V(\Gamma_k)}a_{v,k} - 1$.

Set also $r(C):=\sum_v a_v= \sum_k r_k(C)+ \sum _{v\in \mathcal{N}}a_v$ for the number of components of~$C$.

We apply the surgery formula from Theorem~\ref{thm:surgform} in the rational singularity context for the set
of nodes $\mathcal{N}$ and the cycle $\ZK+s_h\in \ZK+\mathcal{S}'$ in order to deduce the following identity
\begin{equation}\label{eq:surggen}
Q_{[\ZK+s_h]}\,(\ZK+s_h)=Q_{[\ZK+s_h],\mathcal{N}}\,(\ZK+s_h)+
Q^{\Gamma_k}_{[j^*_k(\ZK+s_h)]}(j^*_k(\ZK+s_h)).
\end{equation}
Then the above discussion implies the following surgery type formula for the delta invariant.
\begin{prop}\label{thm:deltasurg}
$\delta(C)=Q_{[\ZK+s_h],\mathcal{N}}\,(\ZK+s_h)+\sum_{k:r_k(C)>0}(r_k(C)-1)$.
\end{prop}

\subsection{The topology of star-shaped graphs}\label{ss:TOPSTAR}

\subsubsection{}
In this section we focus on the topology of singularities $(X,0)$ whose minimal good resolution
graph $\Gamma$ is star-shaped. Let $v_0$ be the central vertex, i.e. the only node of $\Gamma$.

In this section will not assume that the graph is rational.

Note that the link of $(X,0)$ in this case is a Seifert rational homology sphere whose Seifert structure is determined
by its normalized Seifert invariants. Our aim is to express arithmetically the value $Q_{[\ZK+s_h],v_0}(\ZK+s_h)$ of the
one-variable counting function $Q_{[\ZK+s_h],v_0}(t)$ in terms of the Seifert invariants.

\subsubsection{\bf Seifert invariants, lifts and equivariant decomposition of $Z(\mathbf{t})$ (cf. \cite[Section 6]{LNehrhart})}\label{sec:SfZ}
Assume that the star-shaped graph $\Gamma$ has $\nu$ legs. Each leg is a string with normalized Seifert
invariants $(d_i,q_i)$ where $0<q_i<d_i$ and $\gcd(d_i,q_i)=1$, $i=1,\dots,\nu$. That is, if the $i^{th}$ leg
has $s_i$ vertices, say $v_{i1},\dots,v_{is_i}$ with self-intersection numbers $-k_{i1},\dots,-k_{is_i}$ such
that $v_{i1}$ is connected by the central vertex, then they are expressed by the Hirzebruch--Jung continued
fraction $d_i/q_i=[k_{i1},\dots,k_{is_i}]$ ($k_{ij}\geq 2$), see section~\ref{sec:cyclic}. The corresponding
base elements in $L$ will be denoted by $\{E_{ij}\}_{j=1}^{s_i}$. The central vertex $v_0$ with decoration
$-k$ defines the higher-valency element $E_0$ in~$L$.

We use the notation from section~\ref{sec:cyclic}, for example we will write $d^i_{j_1j_2}$ for the determinant
of the subchain of the $i^{th}$ leg connecting the vertices $v_{ij_1}$ and $v_{ij_2}$.

The plumbed 3-manifold associated with such a star-shaped graph $\Gamma$, or, equivalently, the link of $(X,0)$,
has a Seifert structure and the collection of normalized Seifert invariants is denoted by
$Sf=(-k; (d_i,q_i)_{i=1}^{\nu})$. The orbifold Euler number is defined as
$$e=-k+\sum_i q_i/d_i.$$
The negative definiteness of the intersection form is equivalent to~$e<0$.

Fix some $h\in H$. We say that $\ell\in L'$ is a {\it lift} of $h$ if $[\ell_h']=h$; $\ell'_h$
is called a {\it reduced lift} of $h$ if $[\ell_h']=h$ and $\ell'_h$ has the form $\sum _{i=0}^{\nu} c_iE^*_{i}$,
where we abbreviate $E_i:=E_{is_i}$ for simplicity. Note that in general any coefficient of a lift
might depend on the choice of the lift.

Next we list some useful identities regarding the intersection of the dual base elements:
$(E^*_0,E^*_0)=1/e$, $(E^*_0,E^*_i)=1/d_ie$, $(E^*_{ij},E^*_0)=(d^i_{j+1,s_i}E^*_{is_i},E^*_0)$, and
$[E^*_{ij}]=[d^i_{j+1,s_i}E^*_{is_i}]$ for any $1\leq i\leq \nu, 1\leq j\leq s_i$ (see eg.~\cite[11.1]{NemOSZ}
or~\cite[Section 6.1]{LNehrhart}).

For any cycle $\ell'\in L'$ the $E_0$-coefficient of $\ell'$ will be denoted by $\ell'_0:=(\ell',-E^*_0)$.
Furthermore, if $\ell'=c_0E^*_0+\sum_{i,j}c_{ij}E^*_{ij}\in L'$, then we define its {\it reduced transform}
by
\begin{equation} \label{eq:RED} \ell'_{red}:=c_0 E^*_0+\sum_{i,j}c_{ij}\cdot d^i_{j+1,s_i} E^*_i.\end{equation}
Write $\ell'_{red}=\sum_{i=0}^{\nu} c_i E^*_i$.
The above identities show that $[\ell']=[\ell'_{red}]$
in $H$, $\ell'_0=\ell'_{red,0}$, and this coefficient can be expressed as
\begin{equation} \label{sh}
\ell'_{red,0}=\frac{1}{|e|}(c_0+\sum_{i=1}^\nu \frac{c_i}{d_i}).\end{equation}
Thus, for any fixed $h\in H$, if $\ell'_h\in L'$ is a lift of $h$ then its reduced
transform is a reduced lift of $h$ with the same $E_0$-coefficient.

\subsubsection{} \label{sss:hCRIT} {\bf The coefficients of $s_h$.} \

In this subsection we characterize those elements $\ell'=a_0E^*_0+\sum_{i,j}a_{ij}E^*_{ij}\in L'$, which
 are of type $s_h$ for some $h\in H$. This means that if we take $\{a_0, \{a_{ij}\}_{ij}\}$ generic transversal
discs along each irreducible exceptional divisor $\{E_0,\{E_{ij}\}_{ij}\}$, then their collection constitutes
the strict transform of a minimal generic curve. We follow \cite[Prop. 11.5]{NemOSZ}.

Let us consider the reduced transform $\ell'_{red}=\sum_{i=0}^{\nu} a_i E^*_i$ of $\ell'$ as well.

Clearly $\ell'$ determines $\ell'_{red}$ canonically by~\eqref{eq:RED}.
Conversely, we claim that whenever $\ell'$ equals some $s_h$ then
from its reduced transform $\ell'_{red}$ one can recover $\ell'$. First recall that in the cyclic quotient case, for
any $0\leq a<d$, the number $a$ (that is, the class $a[E^*_s]$) determines $s_{a[E^*_s]}=\sum _{i=1}^s a_i E^*_i$ via the algorithms described in section~\ref{sec:cyclic}, see eg.~\eqref{eq:cyclicRid}.

In a very same way, each $a_i$ ($i=1,\ldots, \nu$) determines $\{a_{ij}\}_{j=1}^{s_i}$ by the same algorithms,
eg. by~\eqref{eq:cyclicRid}.

Next we wish to characterize the coefficients of $\ell'_{red}$. By~\cite{NemOSZ},
$\ell'_{red}=\sum_{i=0}^\nu a_iE^*_i$ is the reduced transform of some
$s_h$ if and only if
\begin{equation}\label{eq:CRITRED}
\left\{ \begin{array}{ll} a_0\geq 0, \ d_i>a_i \geq 0 & (1\leq i\leq \nu) \\
1+a_0-kt+\sum_{i=1}^{\nu}\left\lfloor \frac{q_it+a_i}{d_i}\right\rfloor\leq 0 & \mbox{for any } t> 0.
\end{array}\right.
\end{equation}
Note that the last restriction for $t=1$ gives $a_0\leq k-1$ as well.

\subsubsection{}{\bf The topological Poincar\'e series reduced to the variable $v_0$.}
In the sequel, associated with any $h\in H$ we fix a reduced lift $\ell'=\sum_{i=0}^\nu c_i E^*_i$ and we define the
following quasi-linear function
\begin{equation}\label{eq:Da}
\FN_{\mathbf{c}}:\mathbb{Z}\to \mathbb{R}, \ \ \ \FN_{\mathbf{c}}(n):=
1+c_0+kn-\sum_{i=1}^\nu \Big\lceil \frac{q_in -c_i}{d_i}\Big\rceil, \ \ n \in\mathbb{Z},
\end{equation}
where the index vector $\mathbf{c}:=(c_0,\dots,c_\nu)$ reflects for the dependence of $\FN_{\mathbf{c}}$ on the
chosen reduced lift. Then, by~\cite[(6.2.3)]{LNehrhart} (see also \cite{LSzMonoids} for the general case) the
$h$-equivariant part of the topological Poincar\'e series $Z(\mathbf{t})$, reduced to the node $v_0$, can be
expressed as follows
\begin{equation}\label{eq:Zh}
Z_h(t)=\sum_{n \geq -\ell'_0} \max\{0,\FN_{\mathbf{c}}(n )\} \ t^{ n +\ell'_0}.
\end{equation}

\begin{rem}
Consider the trivial class $h=0$ with the obvious reduced lift $\ell'=0$.
Then the function from~\eqref{eq:Da} $\FN_{\mathbf{0}}(n ):=1+kn -\sum_{i=1}^m \lceil q_in /d_i\rceil$
is the \emph{Dolgachev--Pinkham--Demazure} function (see eg. \cite{Pinkham}) targeting the dimension of the
$n $-graded piece of the local algebra of the weighted homogeneous singularity whose resolution graph is~$\Gamma$.
\end{rem}

\subsubsection{\bf Interpretation of $Q_{[\ZK+s_h],v_0}(\ZK+s_h)$}\label{sec:Qcalc}
Now we can apply the previous section to express the counting function $Q_{[\ZK+s_h],v_0}(\ZK+s_h)$.
Fix $h\in H$ and write $s_h=a_0 E^*_0+\sum_{i,j}a_{ij}E^*_{ij}$ as above. We consider its reduced transform
$s_{h,red}=a_0 E^*_0+\sum_{i}a_{i}E^*_{i}$ where $a_i:=\sum_j a_{ij}d^i_{j+1,s_i}$.

The key here is to choose the special reduced lift $\ZK-E+s_{h,red}$ associated with the class $[\ZK+s_h]$,
which by using~\eqref{eq:ZK} reads as
$$\ZK-E+s_{h,red}=(a_0+\nu-2)E^*_0+\sum_i (a_{i}-1)E^*_i.$$
Its $E_0$-coefficient can be computed by the above intersection identities, it is
 $\gamma+s_{h,0}=(1/|e|)(a_0+\nu-2+\sum_i (a_i-1)/d_i)$,
with the notation $\gamma:=Z_{K,0}-1$.
\begin{rem}
Note that $\gamma$ is a key numerical invariant which can be expressed as $\gamma=(\nu-2-\sum_i 1/d_i)/|e|\in \mathbb{Q}$.
It has several names in the literature: the `exponent' of the weighted homogeneous germ $(X,0)$, or $-\gamma$ is called the
\emph{log discrepancy} of $E_0$. $\mathfrak{o}\gamma$ (where $\mathfrak{o}$ is the order of the class of $E^*_0$ in $H$)
also appears as the Goto--Watanabe $a$--invariant (cf. \cite[(3.1.4)]{GW}) of the universal abelian covering of $(X,0)$.
In \cite{neumann.abel} $e\gamma$ appears as the orbifold Euler characteristic (see also \cite[3.3.6]{NOem}).
\end{rem}

By~\eqref{eq:Zh} the corresponding $[\ZK+s_h]$-part of the topological Poincar\'e series is
\begin{equation}\label{eq:Z_ZKs}
\begin{aligned}
Z_{[\ZK+s_h]}(t)&=\sum_{n \geq -\gamma-s_{h,0}} \max\{0,\FN_{\widetilde{\mathbf{a}}}(n)\} \ t^{n +\gamma+s_{h,0}}, \\
\FN_{\widetilde{\mathbf{a}}}(n)&=a_0+\nu-1+kn -\sum_{i=1}^\nu \Big\lceil \frac{q_in -a_i+1}{d_i}\Big\rceil
\end{aligned}
\end{equation}
associated with the index vector $\widetilde{\mathbf{a}}=(a_0+\nu-2,a_1-1,\dots,a_{\nu}-1)$.

Note that the $E_0$--coefficient of $Z_k+s_h$ is one larger than $\gamma+s_{h,0}$.
Therefore, using the definition of the counting function one gets
\begin{equation}\label{eq:Qred}
Q_{[\ZK+s_h],v_0}(\ZK+s_h)=\sum_{-\gamma-s_{h,0}\leq n \leq 0} \max\{0,\FN_{\widetilde{\mathbf{a}}}(n )\}.
\end{equation}
\subsubsection{}\label{sss:Imed}
Let us mention immediately that the summation interval is not empty, that is,
$-\gamma-s_{h,0}\leq 0$ whenever $h\not=0$. Indeed,
$s_{h,0}=(a_0+\sum_ia_i/d_i)/|e|$, cf.~\eqref{sh}.
On the other hand, by \cite{NemOSZ}, $Z_{K,0}=1+\chi/e$, where $\chi$ is the orbifold Euler characteristic $2-\nu +\sum_i 1/d_i$.
Therefore we have to check the inequality
$$a_0+\sum_i a_i/d_i\geq \chi= 2-\nu+\sum_i1/d_i,$$
which follows by an elementary computation, since each $d_i \geq 2$ and at least one
of the member of $\{a_0,a_1,\ldots, a_\nu\}$ is non--zero.

\subsubsection{}
Next, we relate the above formula~\eqref{eq:Qred} with another expression associated with the class $h\in H$.
By the `combinatorial duality' Theorem 4.4.1 of \cite{LNNdual} we have that $Q_{[\ZK]+h,v_0}(\ZK-r_{-h})=\pc(Z_{-h}(t))$,
where ${\rm pc}(Z_{-h}(t))$ denotes the \emph{periodic constant} of the series $Z_{-h}$.
Note that $Q_{[\ZK]+h,v_0}(\ZK-r_{-h})$ counts the coefficients in
$Z_{[\ZK+s_h]}(t)$ with the condition that the exponent $n +\gamma+s_{h,0}$ satisfies
$0\leq n +\gamma+s_{h,0}<Z_{K,0}-r_{-h,0}=\gamma+1-r_{-h,0}$, or,
$-\gamma -s_{h,0}\leq n< 1-r_{-h,0}-s_{h,0}$.

Write $s_{h,0}$ as $r_{h,0}+\Delta_h$ for some $\Delta_h \in \ZZ_{\geq 0}$. Note also that
 $r_{-h,0}=1-r_{h,0}$ if $r_{h,0}\not=0$, otherwise it is zero. Hence $1-r_{-h,0}-s_{h,0}=1-r_{-h,0}-r_{h,0}-\Delta_h$ equals
$-\Delta_h$ if $r_{h,0}\not=0$, and equals $-\Delta_h+1 $ otherwise. It can also be written as $ 1+\lfloor -s_{h,0} \rfloor $.
Hence
\begin{equation}\label{eq:Qred2}{\rm pc}(Z_{-h}(t))
=\sum_{-\gamma-s_{h,0}\leq n \leq \lfloor -s_{h,0} \rfloor} \max\{0,\FN_{\widetilde{\mathbf{a}}}(n )\}.
\end{equation}
Note that the summation in the right-hand side is non--empty only if $-\gamma-s_{h,0}\leq \lfloor -s_{h,0} \rfloor$.
This in the case $r_{h,0}=0$ means $\gamma\geq 0$, and in the case $r_{h,0}\not =0$ means $\gamma+r_h\geq 1$.
(Otherwise ${\rm pc}(Z_{-h}(t))=0$ automatically.)

Note that in general we have the following facts. If $(X,0)$ is ADE then $Z_K=0$, hence $\gamma=-1$. In all other cases $Z_K>0$, hence
$\gamma>-1$. If $(X,0)$ is quotient singularity then all the coefficients of $Z_k$ are less than 1, hence $\gamma<0$.

Summarized, if $-\gamma-s_{h,0}\leq \lfloor -s_{h,0} \rfloor$ then
\begin{equation}\label{eq:Qred3}
Q_{[\ZK+s_h],v_0}(\ZK+s_h)={\rm pc}(Z_{-h}(t))+
\sum_{1+\lfloor -s_{h,0}\rfloor \leq n \leq 0} \max\{0,\FN_{\widetilde{\mathbf{a}}}(n )\}.
\end{equation}
The sum on the right-hand side can also be empty in some cases.

Finally let us comment the expression ${\rm pc}(Z_{-h}(t))$. For any star-shaped graph
${\rm pc}(Z_{-h}(t))$ can be identified with the (normalized)
Seiberg--Witten invariant associated with the link (and spin$^c$ structure corresponding to $-h$). For this identity and expression see eg.
\cite{Nem-CLB} or \cite{LNehrhart}. If the CDGZ identity holds, then it equals the
equivariant geometric genus $ p_g(X,0)_{-h}$ (the $(-h)$--equivariant part of the geometric genus of the universal abelian covering) \cite{Nem-CLB}.
This is true even if the CDGZ identity does not hold, but the equivariant Seiberg--Witten Conjecture holds (see again \cite{Nem-CLB}).
 If $(X,0)$ is rational then it can be replaced by a much simpler combinatorial expression, see below.

We emphasize again that the identities of this subsection \ref{ss:TOPSTAR}
 hold for any star-shaped graph whose plumbing (Seifert) 3--manifold is a rational homology sphere,
and the rationality of $\Gamma$ is not required.

\begin{rem}\label{rem:BOUND}
Using the criterion~\eqref{eq:CRITRED} one verifies that $\FN_{\widetilde{\mathbf{a}}}(n)\leq \nu-2$ for any $n<0$.
\end{rem}

\subsection{The case of rational star-shaped graphs}\label{ss:RATSTAR}

Next, we wish to add some comments regarding the formulae of subsection \ref{ss:TOPSTAR}
in the case when $\Gamma$ is rational.
 By~\ref{sss:LauferCrit}, if $\Gamma$ is rational then necessarily $-e_{v_0}\geq \nu-1$,
where $e_{v_0}$ denotes the self-intersection $E_0^2$ of $E_0$.

In the rational case, the CDGZ identity, or the equivariant Seiberg--Witten Conjecture hold,
hence ${\rm pc}(Z_{-h}(t)) = p_g(X,0)_{-h}$. Furthermore, it can be reinterpreted combinatorially (see
 \cite[Cor. 4.5.3]{NemGR} or \cite[6.4]{LNehrhart}) as
$${\rm pc}(Z_{-h}(t)) = \chi(r_{-h})-\chi(s_{-h}).$$
The identity ${\rm pc}(Z_{-h}(t)) = p_g(X,0)_{-h}$ has the following interpretation too.
Note that $\sum_h p_g(X,0)_{-h}$ is exactly the geometric genus of the universal abelian covering.
$ p_g(X,0)_{0}=p_g(X,0)=0$ by the rationality of $(X,0)$, but the vanishing of the other terms tests whether the universal abelian
covering is rational or not.
(Eg., the rational graph with Seifert invariants
$(-2; (3,1)_{i=1}^{3})$ has non--rational universal abelian covering, cf.
 \cite[Ex. 4.5.4]{NemGR}.)

If $(X,0)$ is a quotient singularity then the universal abelian covering is an ADE germ, hence in this case
${\rm pc}(Z_{-h}(t)) = \chi(r_{-h})-\chi(s_{-h})= p_g(X,0)_{-h}=0$ for any $h\in H$.

\subsection{The delta invariant formula revisited}
 Assume that $\Gamma$ is star-shaped and rational, and we adopt the notations and results of the previous
subsections.

Let us compute the contribution $\FN_{\widetilde{\mathbf{a}}}(0)$. Note that $a_i=0$ if and only if
$\sum _ja_{ij}=0$ if and only if $r_i(C)=0$. Therefore,
$$\FN_{\widetilde{\mathbf{a}}}(0)=
a_0+\nu-1-\sum_{i:r_i(C)\neq 0}\left\lceil \frac{1-a_i}{d_i}\right\rceil-\sum_{i:r_i(C)=0}\left\lceil \frac{1}{d_i}\right\rceil.$$
Since $a_i<d_i$ (cf.~\eqref{eq:CRITRED}), we get $\FN_{\widetilde{\mathbf{a}}}(0)=a_0+\nu-1-\#\{i\,:\, r_i(C)=0\}$.
This identity combined with Proposition \ref{thm:deltasurg} and~\eqref{eq:Qred} provides the following.
\begin{thm}\label{thm:SUM} If $(C,0)$ is a minimal generic curve germ on a rational singularity $(X,0)$ with a star-shaped graph,
$r(C)$ is the number of components of $(C,0)$, and $\sum_{i=0}^\nu a_iE^*_i$ is the reduced transform of $s_h\not=0$, then
\begin{equation}\label{eq:KELL}
\delta(C)=r(C)-1+\sum_{-\gamma-s_{h,0}\leq n \leq -1} \max\{0,\FN_{\widetilde{\mathbf{a}}}(n )\}.
\end{equation}
In particular, $(C,0)$ is an ordinary $r(C)$--tuple if and only if the above sum vanishes.

Furthermore, if $-\gamma-s_{h,0}\leq \lfloor -s_{h,0} \rfloor$ then
\begin{equation}\label{eq:Qred4}
\delta(C)=r(C)-1 +\chi(r_{-h})-\chi(s_{-h})+
\sum_{-\lfloor s_{h,0}\rfloor \leq n \leq -1 } \max\{0,\FN_{\widetilde{\mathbf{a}}}(n )\}.
\end{equation}
\end{thm}
Recall that $\chi(r_{-h})\geq \chi(s_{-h})$.
Hence the above expressions show exactly how far the minimal generic curve numerically from an ordinary $r(C)$-tuple is.

\subsubsection{}\label{sss:FURTHERTOP}{\bf Further topological connections.}
Consider again the series $Z_{[\ZK+s_h]}(t)$ from~\eqref{eq:Z_ZKs}.
The following series is called the {\it polynomial part of $Z_{[\ZK+s_h]}(t)$} in \cite[(6.2.4)]{LNehrhart}:
\begin{equation}\label{eq:Z_ZKs+}
Z^+_{[\ZK+s_h]}(t)=\sum_{n \geq -\gamma-s_{h,0}} \max\{0,-\FN_{\widetilde{\mathbf{a}}}(n)\} \ t^{n +\gamma+s_{h,0}}.
\end{equation}
According to \cite{LNehrhart,NemGR}, since $(X,0)$ is rational,
$Z^+_{[\ZK+s_h]}(1)=p_g(X,0)_{[Z_K]+h}=\chi(r_{[\ZK]+h)})-\chi(s_{[\ZK]+h})$.
Therefore, if the universal abelian covering of $(X,0)$ is rational then
$Z^+_{[\ZK+s_h]}(1)=p_g(X,0)_{[\ZK]+h}=0$, hence in this case
$\FN_{\widetilde{\mathbf{a}}}(n)\geq 0$ for all $n\geq -\gamma-s_{h,0}$.

This fact, for $(X,0)$ quotient singularity, and for the relevant interval $ -\gamma-s_{h,0}\leq n\leq 0$, will be
verified combinatorially in~\ref{sss:SECONDCLAIM}.

\section{Minimal generic curves on quotient singularities}\label{sec:quotient}

\subsection{The general formula}
Let $(X,0)$ be a non-cyclic quotient singularity.

It is well--known that $(X,0)$ is a quotient singularity if and only if $\Gamma$ is {\it numerically log terminal}, that is,
all the coefficients of $Z_K$ are $<1$.

If $\Gamma$ is minimal and star-shaped (but not a string) then $(X,0)$ is quotient if and only if $\nu=3$ and $\sum_i1/d_i>1$.
(Recall that from Laufer algorithm we also have $k\geq 2$.)

In particular, $\gamma\in [-1,0) $ and the inequality $-\gamma-s_{h,0}\leq \lfloor -s_{h,0}\rfloor$ never happens.

Note that $\nu=3$ and
 Remark \ref{rem:BOUND} show that $\max\{0,\FN_{\widetilde{\mathbf{a}}}(n )\}\in \{0,1\}$ for any $n<0$.
In fact, by \ref{sss:FURTHERTOP}, $\max\{0,\FN_{\widetilde{\mathbf{a}}}(n )\}=\FN_{\widetilde{\mathbf{a}}}(n )$
 for $n<0$, a fact which will be reproved in \ref{sss:SECONDCLAIM}.

Then from the previous section we have the delta invariant formula~\eqref{eq:KELL}, subject to the new conditions
imposed by the fact that $(X,0)$ is a non--cyclic quotient singularity.

\subsection{} Let us concentrate first one the summation interval in the formula of~\eqref{eq:KELL} ($h\not=0$).

Note that if $-\gamma-s_{h,0}\leq -1$, then $s_{h,0}\geq 1-\gamma$. Since $\gamma<0$ (because $(X,0)$ is a quotient)
we get that necessarily $s_{h,0}>1$. This proves (via~\eqref{eq:KELL}) the following.

\subsubsection{}\label{sss:FIRSTCLAIM}{\bf Claim 1.} {\it If $s_{h,0}\leq 1$ then $\delta(C)=r(C)-1$, hence $(C,0)$ is an
ordinary $r(C)$--tuple. }\\

\noindent Next, we wish to eliminate `$\max$' from the expression $\max\{0,\FN_{\widetilde{\mathbf{a}}}(n)\}$ for some $n$.

\subsubsection{}\label{sss:SECONDCLAIM}{\bf Claim 2.} {\it If $-\gamma-s_{h,0}\leq -t\leq 0$, then
$\FN_{\widetilde{\mathbf{a}}}(-t )\geq 0$. }

Indeed, using the same identities as in \ref{sss:Imed}, we see that $-\gamma-s_{h,0}\leq -t$
reads as
$$a_0+\sum_ia_i/d_i\geq t |e|+\chi=t(k-\sum_i q_i/d_i) -1+\sum_i1/d_i,$$
or
\begin{equation}\label{eq:A1}
\FP_{\widetilde{\mathbf{a}}}(t)\geq 0, \ \ \\ \mbox{where} \ \ \FP_{\widetilde{\mathbf{a}}}(t)
:=1+a_0-kt+\sum_i\frac{q_it+a_i-1}{d_i}.
\end{equation}
This can/should be compared with the following two expressions. The first one is
\begin{equation}\label{eq:-t}
\FN_{\widetilde{\mathbf{a}}}(-t)=2+a_0-kt + \sum_{i=1}^\nu \Big\lfloor \frac{q_it +a_i-1}{d_i}\Big\rfloor.\end{equation}
The second one is the expression from~\eqref{eq:CRITRED}
\begin{equation}\label{eq:-tt}
\FR_{\widetilde{\mathbf{a}}}(t)=1+a_0-kt + \sum_{i=1}^\nu \Big\lfloor \frac{q_it +a_i}{d_i}\Big\rfloor.\end{equation}
By the minimality assumption and~\eqref{eq:CRITRED}, if $ \widetilde{\mathbf{a}}$ represents some $s_h$, then
 $\FR_{\widetilde{\mathbf{a}}}(t)\leq 0$ for any $t>0$. Also, as we already mentioned (and it can be seen easily),
$\FN_{\widetilde{\mathbf{a}}}(-t)\leq \FR_{\widetilde{\mathbf{a}}}(t) +1$.

We wish to prove that~\eqref{eq:A1} implies 
$\FN_{\widetilde{\mathbf{a}}}(-t)\geq 0$. Write $(q_it+a_i)/d_i$ as $\lfloor (q_it+a_i)/d_i\rfloor + r_i/d_i$
for some $0\leq r_i<d_i$. Then
$$\FN_{\widetilde{\mathbf{a}}}(-t)-\FP_{\widetilde{\mathbf{a}}}(t)=1+\sum_{i: r_i=0} \big( -1+\frac{1}{d_i}\big)+
\sum_{i: r_i\not =0} \big( -\frac{r_i-1}{d_i}\big)=
1+\sum_i \frac{1}{d_i}+\sum_{i: r_i=0} (-1)+
\sum_{i: r_i\not =0} \frac{-r_i}{d_i}.
$$
Since $\sum_i1/d_i>1$ (and we have three legs), the right-hand side is $>-1$, hence $\FN_{\widetilde{\mathbf{a}}}(-t)-\FP_{\widetilde{\mathbf{a}}}(t)>-1$,
or $\FN_{\widetilde{\mathbf{a}}}(-t)>\FP_{\widetilde{\mathbf{a}}}(t)-1\geq -1$. Since it is an integer, we obtain the claim.

Note also that
$\FN_{\widetilde{\mathbf{a}}}(-1)-\FR_{\widetilde{\mathbf{a}}}(1)=1-\#\{i\,:\, r_i=0\}$. Hence we get
\subsubsection{}\label{sss:THIRDCLAIM}{\bf Claim 3.} {\it
$\FN_{\widetilde{\mathbf{a}}}(-1 )=1$ then necessarily $\FR_{\widetilde{\mathbf{a}}}(1)=0$ and
$r_i>0$ for all $i$. }\\

Maybe it is worth to notice that for
 $\FN_{\widetilde{\mathbf{a}}}(-1 )$ we have the upper bound 1,
 and we also proved that $\FN_{\widetilde{\mathbf{a}}}(-1 )\geq 0$ whenever
 $-\gamma-s_{h,0}\leq -1$.
 But, in general, there exists no lower bound, for it, it can be
a very negative number (if $k\ll 0$).

\subsection{}
We can now restate Theorem~\ref{thm:SUM} for the particular case of quotient surface singularities
providing a very explicit formula for the delta invariant of minimal generic curves. The two exceptional
cases where the general formula does not apply are described by their Seifert invariants (see~\ref{sec:SfZ})
and the description of the minimal generic cycle $s_h$ follows the notation in Figures~\ref{fig:exceptional16} 
and~\ref{fig:exceptional17}.

\begin{thm}\label{thm:SUM2:alt}
If $(C,0)$ is a minimal generic curve germ on a quotient singularity $(X,0)$, then
$$
\delta(C)=r(C)-1+\varepsilon(C), \quad \quad \varepsilon(C)\in\{0,1\}.
$$
The explicit value of $\varepsilon(C)$ is given as:
$$
\varepsilon(C)=
\begin{cases}
0 & \text{ if } (X,0) \text{ is cyclic, $\mathbb E_6$, $\mathbb E_7$, or } s_{h,0}\leq 1, \\
1 & \text{ if } (X,0) \text{ is } (-2;(2,1),(3,2),(5,2)) \text{ and } s_h=E_{3,1}^*+E^*_{1,1},\\
1 & \text{ if } (X,0) \text{ is } (-2;(2,1),(3,2),(5,3)) \text{ and } s_h=E^*_{1,1},\\
\FN_{\widetilde{\mathbf{a}}}(-1) & \text{ otherwise,}
\end{cases}
$$
where
\begin{equation}\label{eq:-ta}
\FN_{\widetilde{\mathbf{a}}}(-1)=2+a_0-k + \sum_{i=1}^\nu \Big\lfloor \frac{q_i +a_i-1}{d_i}
\Big\rfloor.
\end{equation}
Moreover, $(C,0)$ is of type $R^r_r$ if $\varepsilon(C)=0$, and of type
$R^{r-1}_r$ if $\varepsilon(C)=1$ (cf. Examples \ref{ex:delta22} and \ref{ex:delta33}).
\end{thm}

\begin{rem}
\mbox{}
\begin{enumerate}
\item
The exceptional cases in Theorem~\ref{thm:SUM2:alt} are special in the sense that
$\FN_{\widetilde{\mathbf{a}}}(-1)=0$
and $\FN_{\widetilde{\mathbf{a}}}(-2)=1$ (see Examples~\ref{sec:exceptional16}\eqref{sec:exceptional16:special}
and~\ref{sec:exceptional17}\eqref{sec:exceptional17:special}).
These are the only exceptions in the quotient singularity case where there is a contribution from
$\FN_{\widetilde{\mathbf{a}}}(n)=0$ for $ n<-1$ (compare with equation~\eqref{eq:KELL}).
\item In~\eqref{eq:-ta} the term $\FN_{\widetilde{\mathbf{a}}}(-1)$ is in general non-trivial,
as we will see several examples in the sequel.
This corresponds in particular to the fact that not any minimal general curve is an ordinary $r(C)$--tuple.

\item On the other hand, the above Theorem has the following consequence: If $(C,0)$ is a minimal
generic curve on $(X,0)$ then any sub-collection of $(C,0)$ with $r(C)-1$ components is an ordinary $(r(C)-1)$--tuple
(since any proper subsets of any $R^r_r$ or $R^{r-1}_r$ is an ordinary $r'$--tuple).
This means that a nontrivial contribution of type $\FN_{\widetilde{\mathbf{a}}}(-1)$  can  appear only  for the
`maximal elements' of type $s_h$ (and not for all of them). (Compare also with \ref{sss:THIRDCLAIM}.)
Maximality here means the following: $s_h$ is maximal if $s_h+E^*_v$ does not equal  $s_{h+[E^*_v]}$ for any $v$.

This fact is hard to see directly. However, we can exemplify it in the following two cases.

First  assume that $\widetilde{\mathbf{a}}$ represents $s_h$ and $s_h+E^*_0=s_{h+[E^*_0]}$ too, that is,
$\widetilde{\mathbf{a}}+E^*_0$ is also a `minimal system'.  Then
$\FN_{\widetilde{\mathbf{a}}}(-t)\leq \FR_{\widetilde{\mathbf{a}}}(t) +1=
\FR_{\widetilde{\mathbf{a}}+E^*_0}(t)\leq 0$ for any $t>0$.

Similarly, if $d_1=2$, and $E_1$ is the $(-2)$--exceptional curve of this leg, we have a similar
statement (with a slightly different proof).
If $\widetilde{\mathbf{a}}$ represents $s_h$ and $s_h+E^*_1=s_{h+[E^*_1]}$, then
$\FN_{\widetilde{\mathbf{a}}}(-t)=
 \FR_{\widetilde{\mathbf{a}}+E^*_1}(t)\leq 0$ for any $t>0$.
\end{enumerate}
\end{rem}

\begin{exam}\label{sec:rational-star-shaped-non-quotient}{\bf A non-quotient rational singularity.}
We end this section with an example, which shows that  if $(X,0)$ is a non-quotient singularity  then
$\delta(C) -r(C)+1$ can be strict larger than one (hence $(C,0)$ is not of type $R^r_r$ or $R^{r-1}_r$).

Assume that the graph of $(X,0)$ is the following.

\begin{figure}[ht]
\begin{tikzpicture}[scale=.4]
\node [shape=circle,fill,draw=black,scale=.6] (O) at (0,0) {}
node [left=5] at (O) {$E_0(-4)$}
node [right=5] at (O) {$E_0^* = (\frac{1}{2},\frac{1}{4},\frac{1}{4},\frac{1}{4},\frac{1}{4})$};
\node[shape=circle,fill, draw=black,scale=.6] (A) at (-3,3) {}
node [left=4] at (A) {$E_1 (-2)$};
\node[shape=circle,fill, draw=black,scale=.6] (B) at (3,3) {}
node [right=5] at (B) {$E_2 (-2)$};
\node[shape=circle,fill, draw=black,scale=.6] (C) at (-3,-3) {}
node [left=5] at (C) {$E_3 (-2)$};
\node[shape=circle,fill, draw=black,scale=.6] (D) at (3,-3) {}
node [right=5] at (D) {$E_4 (-2)$};
\path [-,very thick] (A) edge (0,0);
\path [-,very thick] (B) edge (0,0);
\path [-,very thick] (C) edge (0,0);
\path [-,very thick] (D) edge (0,0);
\end{tikzpicture}
\end{figure}

One verifies that  $Z_K = 2 E_0^*$ and  $\Zmin =E$, hence $(X,0)$ is minimal rational.

Choose $\l' = 3 E_0^* \in \cS'$ and $h = [\l'] \in H$.  The cycle $r_h$ is $\l'-E_0$ and it does not
belong to $\cS'$, so $s_h = \l'$. Let $C$ be the  reduced curve  associated with $s_h$.
Clearly $r(C)=3$.

Note that $\gamma=Z_{K,0}-1=0$, hence  in (\ref{eq:KELL}) $-\gamma-s_{h,0}=-3/2$  and
$\delta(C)=r(C)-1+\max\{0, \FN_{\widetilde{\mathbf{a}}}(-1)\}$.
But a direct verification shows that  $\FN_{\widetilde{\mathbf{a}}}(-1)=2$.

This example shows that $\nu = 3$ is a necessary condition in Theorem~\ref{thm:SUM2:alt}.
In fact, increasing $\nu = k$, one can construct rational star-shaped graphs as above with $(-2)$-legs  and
$\delta(C) - r(C) = k-3$.

\end{exam}

\section{Proof of Theorem \ref{thm:SUM2:alt}}\label{ss:PROOF}
Recall that Theorem~\ref{thm:SUM2:alt} in the cyclic case has already been discussed in section  \ref{sec:cyclic}.
For the non-cyclic quotient singularities the proof is based on~\eqref{eq:KELL},
after showing that, except for the exceptional cases, only $n=-1$ might give a non-zero contribution from
$\FN_{\widetilde{\mathbf{a}}}(n)$. This will be the purpose of the coming section.

\subsection{Preliminary results}
First we will prove two combinatorial lemmas covering different families (with some overlaps) of non-cyclic
quotient singularities. 

\begin{lemma}\label{lem:LEMMA1} If
$(X,0)$ is a non-cyclic quotient singularity, then for any $h\in H$ the following facts hold:
$$\left\{
\begin{array}{l} \mbox{if $k\geq 3$ then $s_{h,0}<2$} \\
\mbox{if $k=2$ and $q_1=q_2=q_3=1$ then $s_{h,0}\leq 2$.} \end{array} \right.$$
\end{lemma}
\begin{proof}
Assume that $s_h=a_0E^*_0+\sum_{ij} a_{ij}E^*_{ij}$. Then by the discussion from \ref{sec:SfZ}
\begin{equation}\label{eq:need}
s_{h,0}=(s_{h,red})_0= \big(a_0+\sum_i a_i/d_i\big)/|e|.
\end{equation}
In the first case we need to show that this expression is $<2$.
Formula~\eqref{eq:Da} for $n=-1$ becomes
$$1+a_0-k +\sum_i \lfloor (q_i+a_i)/d_i\rfloor\leq 0.$$
Note that by~\eqref{eq:CRITRED} $a_i< d_i$.
Let $I$ be the set of indeces when $q_i+a_i \geq d_i$.
Hence the above ineguality reads as
$1+a_0-k+\#I\leq 0$. Using this we have
\begin{equation*}\begin{split}
a_0+\sum_i\frac{a_i}{d_i} & \leq\ k-1-\#I + \sum_{i\not\in I} \frac{d_i-1-q_i}{d_i} +
\Big(\sum_{i\in I} \frac{d_i-1-q_i}{d_i}+\sum_{i\in I}\frac{q_i}{d_i}\Big)\\
 \ & \leq\ k-1 + \sum_{i} \frac{d_i-1-q_i}{d_i}.\end{split}\end{equation*}
Hence it is enough to prove that
\begin{equation}\label{eq:=} k-1 + \sum_{i} \frac{d_i-1-q_i}{d_i}< 2|e|,\end{equation}
or $|e| +2-\sum_i 1/d_i< 2|e|$. This transforms into $2-\sum_i 1/d_i < |e|= k-\sum _i q_i/d_i$.
But $ k-\sum _i q_i/d_i\geq k-\sum_i (d_1-1)/d_i=k-3+\sum_i 1/d_i$, hence it is enough to prove that
$k-3 +\sum_i 1/d_i > 2-\sum_i1/d_i$. This is $k+2\sum_i 1/d_i>5$, which is true for $k\geq 3$ since $\sum_i1/d_i>1$.
 In the second case we repeat the argument and we stop at~\eqref{eq:=}, where we have equallity.
\end{proof}

\begin{exam}
For $k=2$, in general, there exists no such bound for $s_{h,0}$, in fact there exist no universal bound at all for $s_{h,0}$.
Eg., if we take the $D_{d+2}$ graph (with all decorations $-2$), with Seifert invariants
$(2,1), (2,1), (d,d-1)$, and if $E_1$ denotes the end vertex of a $(2,1)$ leg, then $E^*_1=s_{[E^*_1]}$ and $s_{h,0}=d/2$,
a number which can be arbitrarily large.
\end{exam}

\begin{lemma} \label{lem:LEMMA2}
 If
$(X,0)$ is a non-cyclic quotient singularity with $k=2$ and $q_1=q_2=1$ then $\FN_{\widetilde{\mathbf{a}}}(n )\leq 0$ for any $n<-1$.
\end{lemma}
\begin{proof}
We compare $\FN_{\widetilde{\mathbf{a}}}(-(t+1))$ with $\FR_{\widetilde{\mathbf{a}}}(t)$.
Note that if we use for the third floor expression
$\lfloor x+(q_3-1)/d_3\rfloor\leq \lfloor x\rfloor +1$ we get
$\FN_{\widetilde{\mathbf{a}}}(-(t+1))\leq \FR_{\widetilde{\mathbf{a}}}(t)\leq 0$.
\end{proof}

\subsection{The non-cyclic case}
To end the proof, first assume that $k\geq 3$. Then by Lemma~\ref{lem:LEMMA1} $s_{h,0}<2$, hence
$-\gamma-s_{h,0}>-2$, hence the summation in~\eqref{eq:KELL} reduces to the term corresponding to $n=-1$.
Then, by Claim 2 from \ref{sss:SECONDCLAIM}, we have $\FN_{\widetilde{\mathbf{a}}}(-1)\geq 0$ as well.

By Laufer's criterion $k\geq \nu-1=2$, hence the only remaining case is $k=2$, what we will assume next.

If $q_1=q_2=1$, then by Lemma \ref{lem:LEMMA2} we conclude again that only the term $n=-1$ contributes.
Note that this case includes all the infinite family of Seifert invariants $(-2; (2,1),(2,1), (d_3,q_3))$.

Next assume that $\Gamma$ is one of the classical $\mathbb E_6$, $\mathbb E_7$, and $\mathbb E_8$ graphs.
Their groups $H$ are $\ZZ_3$, $\ZZ_2$, and $\ZZ_1$ respectively (hence the case $\mathbb E_8$ contains no
non--zero $s_h$).
In the first two cases one verifies easily that each non--zero $s_h$ is represented by an
irreducible curve supported in an end exceptional divisor, along which the multiplicity of the minimal cycle
$\Zmin$ is one. This means that the irreducible curve $(C,0)$ (embedded in $(X,0)\subset (\CC^e,0)$, where
$e=3$ is the embedded dimension of $(X,0)$) is intersected by the generic hyperplane with intersection
multiplicity 1. This means that $(C,0)$ is smooth, hence $R^1_1$.
(The statements regarding these examples can also be verified directly.)

In this way we covered all cases, except for the singularities with Seifert invariants
$(-2; (2,1), (3,2), (5,3))$ and $(-2; (2,1), (3,2), (5,2))$.
There are verified below.

\subsection{The exceptional singularity $(-2; (2,1), (3,2), (5,2))$}\label{sec:exceptional16}
In this case one has $H=\ZZ/13\ZZ[E_0^*]$.
The  canonical cycle is
$Z_K = E_{3,1}^*$ with $Z_{K,0}=12/13$.   

Consider any representative of a class $h\in H$, say $hE^*_0$.
Using the generalized Laufer's algorithm one obtains the minimal cycle $s_h$.

\begin{figure}[ht]
\begin{tikzpicture}[scale=.4]
\node [shape=circle,fill,draw=black,scale=.6] (O) at (0,0) {}
node [below left] at (O) {$E_0$}
node [below right] at (O) {$(-2)$};
\node[shape=circle,fill, draw=black,scale=.6] (A) at (-4,0) {}
node [above=4] at (A) {$E_{3,1} (-3)$};
\node[shape=circle,fill, draw=black,scale=.6] (AA) at (-8,0) {}
node [below=4] at (AA) {$E_{3,2} (-2)$};
\node[shape=circle,fill, draw=black,scale=.6] (B) at (4,0) {}
node [above=4] at (B) {$E_{2,1} (-2)$};
\node[shape=circle,fill, draw=black,scale=.6] (C) at (0,-3) {}
node [left=5] at (C) {$E_{1,1} (-2)$};
\node[shape=circle,fill, draw=black,scale=.6] (BB) at (8,0) {}
node [below=4] at (BB) {$E_{2,2} (-2)$};
\path [-,very thick] (A) edge (0,0);
\path [-,very thick] (AA) edge (A);
\path [-,very thick] (B) edge (0,0);
\path [-,very thick] (B) edge (BB);
\path [-,very thick] (C) edge (0,0);
\end{tikzpicture}
\caption{}
\label{fig:exceptional16}
\end{figure}

Note that if
$s_h = a_0E_0^*+  \sum_{i,j} a_{i,j} E_{i,j}^*$, then
$$
\begin{aligned}
s_{h,red}
&= \sum_{i=0}^3 a_i E_i^*= a_0 E_0^* + a_{1,1} E_1^* + (2a_{2,1} + a_{2,2}) E_2^* + (2a_{3,1}+a_{3,2}) E_3^*,
\end{aligned}
$$
where $E_1^*=E_{1,1}^*$, $E_2^*=E_{2,2}^*$, $E_3^*=E_{3,2}^*$   
(see ~\ref{sss:hCRIT}) and hence $\FN_{\widetilde{\mathbf{a}}}(-t)$ equals
\begin{equation}
\label{eq:Nn}
 2 + a_0 - 2t + \Big\lfloor \frac{t+a_{1,1}-1}{2} \Big\rfloor
+ \Big\lfloor \frac{2t+2a_{2,1}+a_{2,2}-1}{3} \Big\rfloor
+ \Big\lfloor \frac{2t+2a_{3,1}+a_{3,2}-1}{5} \Big\rfloor.
\end{equation}

The values of the delta invariant for the minimal generic $h$-curves $C_h$ for $h\neq 0$ are given below.
They are  obtained from Theorem~\ref{thm:SUM}.
Each  $\kappa(s_h)$ is  obtained using Theorem~\ref{thm:mainkappa}.
The values $\FN_{\widetilde{\mathbf{a}}}(n)$ for $-\lfloor s_{h,0} + \gamma \rfloor\leq n \leq -1$
are obtained explicitly from $s_{h,red}$ and~\eqref{eq:Nn}.
\begin{enumerate}
 \item $s_{h,0}\leq 1$, $\lfloor s_{h,0} + \gamma \rfloor = 0$ for:
$$
\array{lll}
\delta(C_3)=\kappa(E^*_{3,1})=0, & \delta(C_8)=\kappa(E^*_{3,2})=0, & \delta(C_9)=\kappa(E^*_{2,2})=0.
\endarray
$$
 \item $s_{h,0}> 1$, $\lfloor s_{h,0} + \gamma \rfloor = 1$,
$\FN_{\widetilde{\mathbf{a}}}(-1) = 0$ for:
$$
\begin{array}{lll}
\delta(C_2)=\kappa(E^*_{3,2}+E^*_{1,1})=1, &
\delta(C_{4})=\kappa(E^*_{3,2}+E^*_{2,2})=1, & \delta(C_7)=\kappa(E^*_{1,1})=0,
\\[0.1cm]
\delta(C_{11})=\kappa(E^*_{3,1}+E^*_{3,2})=1, &
\delta(C_{12})=\kappa(E^*_{3,1}+E^*_{2,2})=1.
\end{array}
$$
 \item $s_{h,0}> 1$, $\lfloor s_{h,0} + \gamma \rfloor = 1$,
$\FN_{\widetilde{\mathbf{a}}}(-1) = 1$ for:
$$
\array{lll}
\delta(C_5)=\kappa(E^*_{2,1})=1, & \delta(C_6)=\kappa(2E^*_{3,1})=2.
\endarray
$$
 \item $s_{h,0}> 1$, $\lfloor s_{h,0}+\gamma\rfloor=2$,
 $\FN_{\widetilde{\mathbf{a}}}(-1)=1$, $\FN_{\widetilde{\mathbf{a}}}(-2)=0$ for:
$$\delta(C_1)=\kappa(E^*_0)=1.$$
 \item\label{sec:exceptional16:special}
 $s_{h,0}> 1$, $\lfloor s_{h,0}+\gamma\rfloor=2$,
 $\FN_{\widetilde{\mathbf{a}}}(-1)=0$, $\FN_{\widetilde{\mathbf{a}}}(-2)=1$ for:
$$\delta(C_{10})=\kappa(E^*_{3,1}+E^*_{1,1})=2.$$
\end{enumerate}

Note that the classes of the irreducible cuts are the following:
$[E_0^*]=1$, $[E_{1,1}^*]=7$, $[E_{2,1}^*]=5$, and $[E_{2,2}^*]=9$, $[E_{3,1}^*]=3$, $[E_{3,2}^*]=8$.

The curves $C_3$, $C_7$, $C_8$, and $C_9$ are smooth irreducible minimal generic, that is, of type $R_1^1$.
The curve $C_1$ and $C_5$ are singular irreducible minimal generic $h$-curves, that is, of type $R_1^0$,
As for the rest, $C_6$ and $C_{10}$ are not irreducible with smooth tangent components, that is, of type
$R_2^1$, see~\ref{ex:delta33} and the remaining ones are nodal, that is, of type~$R_2^2$.

\subsection{The exceptional singularity $(-2; (2,1), (3,2), (5,3))$}\label{sec:exceptional17}
Similarly, one has $H=\ZZ/7\ZZ[E_0^*]$.
The canonical cycle is
$Z_K = E_{3,2}^*$ with $Z_{K,0}=6/7$.

\begin{figure}[ht]
\begin{tikzpicture}[scale=.4]
\node [shape=circle,fill,draw=black,scale=.6] (O) at (0,0) {}
node [below left] at (O) {$E_0$}
node [below right] at (O) {$(-2)$};
\node[shape=circle,fill, draw=black,scale=.6] (A) at (-4,0) {}
node [above=4] at (A) {$E_{3,1} (-2)$};
\node[shape=circle,fill, draw=black,scale=.6] (AA) at (-8,0) {}
node [below=4] at (AA) {$E_{3,2} (-3)$};
\node[shape=circle,fill, draw=black,scale=.6] (B) at (4,0) {}
node [above=4] at (B) {$E_{2,1} (-2)$};
\node[shape=circle,fill, draw=black,scale=.6] (C) at (0,-3) {}
node [left=5] at (C) {$E_{1,1} (-2)$};
\node[shape=circle,fill, draw=black,scale=.6] (BB) at (8,0) {}
node [below=4] at (BB) {$E_{2,2} (-2)$};
\path [-,very thick] (A) edge (0,0);
\path [-,very thick] (AA) edge (A);
\path [-,very thick] (B) edge (0,0);
\path [-,very thick] (B) edge (BB);
\path [-,very thick] (C) edge (0,0);
\end{tikzpicture}
\caption{}
\label{fig:exceptional17}
\end{figure}

The values of the delta invariant for the minimal generic $h$-curves $C_h$ for $h\neq 0$ are given below.
Note that here the expression of $\FN_{\widetilde{\mathbf{a}}}(-t)$ becomes:
$$
\FN_{\widetilde{\mathbf{a}}}(-t)
= 2 + a_0 - 2t + \Big\lfloor \frac{t+a_{1,1}-1}{2} \Big\rfloor
+ \Big\lfloor \frac{2t+2a_{2,1}+a_{2,2}-1}{3} \Big\rfloor
+ \Big\lfloor \frac{3t+3a_{3,1}+a_{3,2}-1}{5} \Big\rfloor.
$$
\begin{enumerate}
\item $s_{h,0}\leq 1$, $\lfloor s_{h,0}+\gamma\rfloor=0$ for:
$$
\delta(C_3)=\kappa(E^*_{3,2})=0.
$$

\item $s_{h,0}> 1$, $\lfloor s_{h,0}+\gamma\rfloor=1$, $\FN_{\widetilde{\mathbf{a}}}(-1)=0$ for:
$$
\delta(C_5)=\kappa(E^*_{2,2})=0, \qquad \delta(C_6)=\kappa(2E^*_{3,2})=1.
$$

\item $s_{h,0}> 1$, $\lfloor s_{h,0}+\gamma\rfloor=2$,
$\FN_{\widetilde{\mathbf{a}}}(-1)=0$,
$\FN_{\widetilde{\mathbf{a}}}(-2)=0$ for:
$$
\delta(C_1)=\kappa(E^*_{3,2}+E^*_{2,2})=1.
$$

\item $s_{h,0}> 1$, $\lfloor s_{h,0}+\gamma\rfloor=2$,
$\FN_{\widetilde{\mathbf{a}}}(-1)=1$,
$\FN_{\widetilde{\mathbf{a}}}(-2)=0$ for:
$$
\delta(C_2)=\kappa(E^*_{3,1})=1.
$$

\item\label{sec:exceptional17:special} $s_{h,0}> 1$, $\lfloor s_{h,0}+\gamma\rfloor=2$,
$\FN_{\widetilde{\mathbf{a}}}(-1)=0$,
$\FN_{\widetilde{\mathbf{a}}}(-2)=1$ for:
$$
\delta(C_4)=\kappa(E^*_{1,1})=1.
$$
\end{enumerate}

The classes of the irreducible cuts are the following:
$[E_0^*]=1$, $[E_{1,1}^*]=4$, $[E_{2,1}^*]=[E_{3,2}^*]=3$, $[E_{2,2}^*]=5$, and $[E_{3,1}^*]=2$.\\

This ends the proof of the theorem.

\end{document}